\documentclass[12pt]{amsart}
\usepackage{amssymb}
\usepackage{amsfonts}
\usepackage{times}
\usepackage{cancel}
\usepackage{bbding}
\usepackage{mathptmx}
\usepackage{amsmath}
\usepackage[usenames]{color}
\usepackage{mathrsfs}
\usepackage{stmaryrd}
\usepackage{amsfonts}
\usepackage{amscd}
\usepackage{cite}
\usepackage{cases}
\usepackage{amsthm}
\usepackage[all]{xy}           
\usepackage{bbm}
\usepackage{txfonts}
\usepackage{amscd}
\usepackage{tikz}
\usetikzlibrary{matrix}
\usepackage[shortlabels]{enumitem}

\usepackage{ifpdf}
\ifpdf
\usepackage[colorlinks,final,backref=page,hyperindex]{hyperref}
\else
\usepackage[colorlinks,final,backref=page,hyperindex]{hyperref}
\fi
\usepackage{tikz}
\usepackage[active]{srcltx}

\makeatletter

\def\la{\lambda}

\def\p{\partial}
\def\vs{\vspace*}
\def\Z{\mathbb{Z}}

\def\A{\mathcal{A}}
\def\T{\mathfrak{T}}

\def\O{\mathcal{O}}
\def\B{\mathcal{B}}
\def\C{\mathbb{C}}

\def\vs{\vspace*}

\newcommand {\emptycomment}[1]{}

\numberwithin{equation}{section}
\newtheorem{theo}{Theorem}[section]
\newtheorem{defi}[theo]{Definition}
\newtheorem{coro}[theo]{Corollary}
\newtheorem{lemm}[theo]{Lemma}
\newtheorem{prop}[theo]{Proposition}

\newtheorem{ex}[theo]{Example}

\newtheorem{remark}[theo]{Remark}

\newcommand{\ad}{\mathrm{ad}}
\newcommand{\NS}{\mathrm{NS}}
\newcommand{\Nat}{\mathbb N}
\newcommand{\perm}{\mathbb S}
\newcommand{\Id}{\rm{Id}}
\newcommand{\g}{\mathfrak g}
\newcommand{\fad}{\mathfrak{ad}}

\newcommand{\NR}{\mathrm{NR}}

\newcommand{\dM}{\mathbf{d}}
\newcommand{\half}{\frac{1}{2}}
\allowdisplaybreaks[4]
\begin{document}

\title[]{Twisting theory,  relative Rota-Baxter type operators and $L_\infty$-algebras on Lie conformal algebras}

\author{Lamei Yuan}
\address{School of Mathematics, Harbin Institute of Technology, Harbin
	150001, China}
\email{lmyuan@hit.edu.cn}

\author{Jiefeng Liu$^{\ast}$}
\address{School of Mathematics and Statistics, Northeast Normal University, Changchun 130024, China}
\email{liujf534@nenu.edu.cn}
\thanks{$^{\ast}$ the corresponding author}
\begin{abstract}
Based on Nijenhuis-Richardson bracket and bidegree on the cohomology complex for a Lie conformal algebra, we develop a twisting theory of Lie conformal algebras. By using derived bracket constructions, we construct $L_\infty$-algebras from (quasi-)twilled Lie conformal algebras. And we show that the result of the twisting by a $\C[\partial]$-module homomorphism on a (quasi-)twilled Lie conformal algebra is also a (quasi-)twilled Lie conformal algebra if and only if the $\C[\partial]$-module homomorphism is a Maurer-Cartan element of the $L_\infty$-algebra. In particular, we show that relative Rota-Baxter type operators on Lie conformal algebras are Maurer-Cartan elements. Besides, we propose a new algebraic structure, called NS-Lie conformal algebras, that is closely  related to twisted relative Rota-Baxter operators and Nijenhuis operators on Lie conformal algebras. As an application of twisting theory, we give the cohomology of  twisted relative Rota-Baxter operators and study their deformations.
\end{abstract}



\maketitle

\tableofcontents

\allowdisplaybreaks
\section{Introduction}

The notion of a Lie conformal algebra was introduced by Kac in \cite{KAC} as an algebraic language to encode the properties of operator product expansions in conformal field theory, and, at the same time, of local Poisson brackets in the theory of integrable evolution equations \cite{BDK2}.
The general structure theory,
cohomology theory and representation theory for conformal algebras have been established and widely developed in the literatures (see, for example, \cite{BDK,BKV,DK3,DK}).

Rota-Baxter operators on associative algebras were initially introduced by Baxter \cite{B} in the study of the fluctuation theory in probability, and then popularized by Rota \cite{Ro}, Atkinson \cite{A} and Cartier \cite{C} during the process of finding their interrelations with combinatorics. Rota-Baxter algebras have broad connections with mathematical physics, including  the application in Connes-Kreimer's algebraic approach to the renormalization in perturbative quantum field theory \cite{CK},  noncommutative symmetric functions and Hopf algebras \cite{GJH,Yu-Guo},  splitting of operads \cite{BBGN,PBG}, quantum analogue of Poisson geometry \cite{U1} and double Poisson algebras \cite{ARR,G3,Goncharov}. We refer the reader to \cite{Gub} for more details about Rota-Baxter algebras.

In the Lie algebra context, a Rota-Baxter operator was introduced independently in the 1980s as the operator form of the classical Yang-Baxter equation, named after the physicists C.-N. Yang and R. Baxter \cite{Bax,Yang}, whereas the classical Yang-Baxter equation plays important roles in mathematics and mathematical physics such as integrable systems and quantum groups \cite{CP,STS}.  In order to gain better understanding of the interrelation between the classical Yang-Baxter equation and the related integrable systems, the more general notion of a relative Rota-Baxter operator (also called $\mathcal O$-operator) on Lie algebras was introduced by Kupershmidt \cite{Ku}.	Recently, Rota-Baxter operators are defined in the categories of Lie groups \cite{GLS} and cocommutative Hopf algebras \cite{Gon}. Besides, cohomology, deformations, extensions and homotopy theory of Rota-Baxter operators were well studied in \cite{JS,LST,TBGS}.

Motivated by the study of conformal analogue of Lie bialgebras, Liberati developed the theory of Lie conformal bialgebras in \cite{Li}. The notion of conformal classical Yang-Baxter equation was introduced to construct coboundary Lie conformal bialgebras. In order to study the operator forms of solutions of the conformal classical Yang-Baxter equation, the authors in \cite{HB1} introduced the notion of relative Rota-Baxter operators on Lie conformal algebras. See \cite{LZ,HB1,YUAN} for further details of relative Rota-Baxter operators on Lie and associative conformal algebras. The twisting theory was introduced by Drinfeld in \cite{D90} motivated by the study of quasi-Lie bialgebras and quasi-Hopf algebras. As a useful tool in the study of bialgebras, the twisting theory was  applied to  Poisson geometry and associative algebras  by Kosmann-Schwarzbach \cite{KS05}, Roytenberg \cite{Roy02A} and Uchino \cite{U2}.

In this paper, we develop the twisting theory of Lie conformal algebras and use this theory to study relative Rota-Baxter type operators. By the derived bracket constructions of Lie conformal algebras, we construct an  $L_\infty$-algebra  associated to a quasi-twilled Lie conformal algebra and show that Rota-Baxter type operators are  Maurer-Cartan elements. We introduce the notion of NS-Lie conformal algebras, which is  the underlying algebraic structure of the  twisted relative Rota-Baxter operators and Nijenhuis operators on a Lie conformal algebra. As an  application of twisting theory of Lie conformal algebras, we give the cohomology of twisted relative Rota-Baxter operators and study their deformations.

In Section 2, we recall the definitions of Lie conformal algebras and their modules.  We also review the cohomology theory of Lie conformal algebras and gather some facts.

In Section 3, we first recall Nijenhuis-Richardson bracket $[-,-]_\NR$ on the cohomology complex $C^*(\A,\A)$ for a Lie conformal algebra $\A$ and thus obtain a graded Lie algebra $(C^*(\A,\A),[-,-]_\NR)$. Next, assuming $\A=\A_1\oplus \A_2$ has a decomposition into a direct sum of two $\C[\partial]$-modules $\A_1$ and $\A_2$, we propose a bidegree on $C^*(\A_1\oplus \A_2,\A_1\oplus \A_2)$. We will see that the Lie conformal algebra structure $\Pi$ of $\A$ is decomposed into the unique four substructures
\begin{align}\label{intro}
\Pi=\hat{\phi}_1+\hat{\mu}_1+\hat{\mu}_2+\hat{\phi}_2.
\end{align}
When $\hat{\phi}_2=0,$ $\A_2$ is a subalgebra of $\A$. We call  such a triple $(\A,\A_1, \A_2)$ a quasi-twilled Lie conformal algebra. When $\hat{\phi}_1=\hat{\phi}_2=0,$ namely, both $\A_1$ and $\A_2$ are  subalgebras of $\A$, the triple $(\A,\A_1, \A_2)$ is called a twilled Lie conformal algebra and denoted by $\A_1\bowtie\A_2$. By the derived bracket construction in \cite{KS}, we show that the quasi-twilled Lie conformal algebra structure on $\A_1\oplus \A_2$ induces an $L_\infty$-algebra structure on $C^*(\A_2,\A_1)$
(see Theorem \ref{quasi-as-shLie}).

In Section 4, we define a twisting operation associated to a $\C[\partial]$-module homomorphism $H:\A_2\rightarrow\A_1$ as $\Pi^H:=e^{X_{\hat{H}}}(\Pi),$ where ${\hat{H}}$ is the lift of $H$. The twisting operation $\Pi^H$ is also decomposed into the unique four substructures corresponding to those of $\Pi$ in \eqref{intro}. We will present explicit formulas of the transformation rules (see Theorem \ref{thm:twist}).

In the twilled case, we show that $(\A_1\bowtie\A_2,\Pi^H)$ is again a twilled Lie conformal algebra if and only if $H$ is a Maurer-Cartan element of the differential graded Lie algebra $(C^*(\A_2,\A_1),d_{\hat{\mu}_2},[-,-]_{\hat{\mu}_1})$ (see Proposition \ref{twisting-twilled}), namely,
\begin{equation*}
d_{\hat{\mu}_2}\hat{H}+\half[\hat{H},\hat{H}]_{\hat{\mu}_1}=0.
\end{equation*}
Furthermore, we show that $T$ is a relative Rota-Baxter operator if and only if $\hat{T}$ is a  Maurer-Cartan element of a  graded Lie algebra (see Proposition \ref{pro:gLa-RB}). As an application, we obtain that $r=\sum_{i} a_i\otimes b_i\in \A\otimes \A$ is a skew-symmetric solution of the  conformal classical Yang-Baxter equation if and only if $r^{\sharp}_{0}=r_{\lambda}^{\sharp}|_{\lambda=0}$ is a  Maurer-Cartan  element of the graded Lie algebra $(C^*(\A^{*c},\A),[\cdot,\cdot]_{\hat{\mu}})$ (see Proposition \ref{44}).

In the quasi-twilled case, we introduce the notion of a twisted relative Rota-Baxter operator, which is a generalization of Rota-Baxter operator and characterized by a $2$-cocycle. We show that basic properties of a relative Rota-Baxter operator are satisfied for twisted one. Moreover, we introduce Nijenhuis operators and Reynolds operators to illustrate examples of twisted relative Rota-Baxter operators.

In Section 5, we construct a new algebraic structure, called NS-Lie conformal algebras. We show that NS-Lie conformal algebras connect closely with Lie conformal algebras, conformal NS-algebras, twisted relative Rota-Baxter operators and Nijenhuis operators. The corresponding results are stated in Theorems \ref{p5}--\ref{th7} and Proposition \ref{p6}.

In Section 6, we introduce cohomology of twisted relative Rota-Baxter operators
and study infinitesimal deformations of them from cohomological point of view.

Throughout this paper, all vector spaces, linear maps and tensor products are over the complex field $\C$. 

\section{Preliminaries}

In this section, we recall the notions of Lie conformal algebras, conformal modules and cohomology. Also, we gather some known results for later use. The material can be found in the literatures \cite{BKV,DK,DK3,KAC,HB1}.

A {\bf conformal algebra} $\mathcal{A}$ is a $\mathbb{C}[\partial]$-module endowed with a $\mathbb{C}$-linear map
$\mathcal{A}\otimes \mathcal{A} \rightarrow \mathcal{A} [\lambda], ~a\otimes b\mapsto a_\lambda b,$
satisfying conformal sesquilinearity:
\begin{eqnarray}\label{sesquilinearity1}
(\partial a)_\lambda b=-\lambda a_\lambda b,~~~
a_\lambda \partial b=(\partial+\lambda)a_\lambda b, ~~\forall ~~ a, b\in \mathcal{A}.
\end{eqnarray}
If, in addition, it satisfies associativity:
\begin{eqnarray}\label{ASS}
(a_\lambda b)_{\lambda+\mu}c=a_\lambda (b_\mu c), ~~\forall ~~ a, b\in \mathcal{A},
\end{eqnarray}
then $\mathcal{A}$ is called an {\bf associative conformal algebra}.

A conformal algebra $\mathcal{A}$ is called {\bf finite} if it has finite rank as $\mathbb{C}[\partial]$-module. By conformal sesquilinearity, the following equalities hold ($a, b, c\in \A$):
\begin{align}\label{NS3}
(a_{-\la-\p}b)_{\la+\mu}c=(a_\mu b)_{\la+\mu} c,\ \
a_{\mu}(b_{-\la-\p}c)=a_{\mu}(b_{-\la-\mu-\p}c).
\end{align}
Changing the variables in \eqref{NS3} gives
\begin{align}\label{NS4-4}
(a_{-\mu-\p}b)_{-\la-\p}c=(a_{-\la-\mu-\p}b)_{-\la-\p}c,\ \
a_{-\la-\mu-\p}(b_{-\la-\p}c)=a_{-\la-\mu-\p}(b_\mu c).
\end{align}	

\begin{defi}{\rm
		A {\bf Lie conformal algebra} $\A$ is a $\C[\partial]$-module endowed with a $\C$-linear map
		$ \A\otimes \A\rightarrow \A[\lambda],~ a\otimes b \mapsto [a_\lambda b],$
		called the $\la$-bracket, and
		satisfying the following axioms for all $a, b, c\in \A$:
		
		Conformal sesquilinearity: $[\partial a_\lambda b]=-\lambda[a_\lambda b],\ \ [ a_\lambda \partial b]=(\partial+\lambda)[a_\lambda b]$,
		
		Skew-symmetry: ~~~~~~~~~~~~~~~~ ${[a_\lambda b]} = -[b_{-\lambda-\partial}a]$,
		
		Jacobi  identity: ~~~~~~~~~~~~~~~~~ ${[a_\lambda[b_\mu c]]}=[[a_\lambda b]_{\lambda+\mu}c]+[b_\mu[a_\lambda c]]$.
	}\end{defi}

	Let $M$ and $N$ be $\mathbb{C}[\partial]$-modules. A {\bf conformal linear map} from $M$ to $N$ is a $\mathbb{C}$-linear map
	$f_\lambda: M\rightarrow N[\lambda]$ satisfying
	$f_\lambda(\partial u)=(\partial+\lambda)f_\lambda u,$ for $u\in M.$
	The vector space of all such maps, denoted by ${\rm Chom}(M,N)$, is a $\C[\partial]$-module via:
	\begin{align*}
	(\partial f)_\la=-\la f_\la, \ \mbox{for}\ \ f_\la\in {\rm Chom}(M,N).
	\end{align*}
	Define the {\bf conformal dual} of a $\C[\partial]$-module $M$ as $M^{*c}={\rm Chom}(M,\C)$, where $\C$ is viewed as the trivial $\C[\partial]$-module, that is,
	$ M^{*c}= \{a:M\rightarrow \C[\lambda] \mid  \mbox{$a$ is $\C$-linear and } a_{\lambda}(\partial b)=\lambda a_{\lambda}b   \}.$
	If $M$ is a finitely generated $\C[\partial]$-module, then ${\rm Cend}(M):={\rm Chom}(M, M)$
	is an associative conformal algebra via:
	\begin{align*}
	(f_\la g)_\mu v=f_\la(g_{\mu-\la}v), ~~ \mbox{for} ~~v\in M, ~ f,g\in {\rm Cend}(M).
	\end{align*}
	Hence ${\rm Cend}(M)$ becomes a Lie conformal algebra, denoted by ${\rm gc}(M)$ and called the {\bf general Lie conformal algebra} on $M$, with respect to
	the following $\la$-bracket:
	\begin{align*}
	[f_\la g]_\mu v=f_\la(g_{\mu-\la}v) -g_{\mu-\la}(f_\la v),~~ \mbox{for} ~~v\in M, ~ f,g\in {\rm Cend}(M).
	\end{align*}
	
	Hereafter all $\C[\partial]$-modules are assumed to be finitely generated.	
	\begin{defi}\label{def1}{\rm   A $\mathbb{C}[\partial]$-module $M$ is called a {\bf module} of a Lie conformal algebra $\A$
			if there is a $\mathbb{C}$-linear map $\rho:\A\rightarrow {\rm Cend}(M)$, such that
			\begin{align}
			\rho(a)_\lambda\rho(b)_\mu-\rho(b)_\mu\rho(a)_\lambda=[\rho(a)_\lambda\rho(b)]_{\lambda+\mu}=\rho([a_\lambda b])_{\lambda+\mu}, ~~\rho(\partial(a))_\lambda=-\lambda\rho(a)_\lambda, ~~\forall ~a,b\in\A.\label{module}
			\end{align}That is, $\rho$ is a homomorphism of Lie conformal algebras from $\A$ to ${\rm gc}(M)$.}
	\end{defi}
	
	For convenience, we will denote a module $M$ of the Lie conformal algebra $\A$ by $(M;\rho)$.
	It is straightforward to check that the following identities hold for all $a,b\in \A$:
	\begin{align}
	\rho([a_\lambda b])_{-\p-\mu}&=\rho(a)_\lambda \rho(b)_{-\p-\mu} -\rho(b)_{-\p-\la-\mu} \rho(a)_\lambda, \label{module1}\\
	\rho([a_\mu b])_{-\p-\la} & =\rho(a)_\mu \rho(b)_{-\p-\la} -\rho(b)_{-\p-\la-\mu} \rho(a)_{-\p-\lambda}. \label{module2}
	\end{align}
	It is not hard to check that $(M^{*c};\rho^*)$ is a module of $\A$, where  $\rho^{*}:\A\rightarrow {\rm gc}(M^{*c})$ is defined by
	\begin{equation*}
	(\rho^{*}(a)_{\lambda}\varphi)_{\mu}v=-\varphi_{\mu-\lambda}(\rho(a)_{\lambda}v), ~\mbox {for}~a\in \A, ~\varphi\in M^{*c}, ~v\in M.
	\end{equation*}
	Define $\ad:\A\rightarrow {\rm gc}(\A)$ by $\ad(a)_\lambda b=[a_{\lambda}b]$ for $a,b\in \A$. Then  $(\A;\ad)$ is a module of $\A$, called the {\bf adjoint module}.  Hence, $(\A^{*c};\ad^*)$ is also a module of $\A$, called the {\bf coadjoint module}.

	\begin{prop}{\rm(\cite{DK3})} \label{p1}Given a Lie conformal algebra $\mathcal{A}$ and an $\mathcal{A}$-module $(M;\rho)$, the vector space $\mathcal{A}\oplus M$ is a $\C[\partial]$-module via $\partial(a,m)=(\partial^{\A} a,\partial^M m)$ and then carries a Lie conformal algebra structure given by
		\begin{align*}
		[(a,m)_\lambda(b,n)]=([a_\lambda b],\rho(a)_\lambda n-\rho(b)_{-\partial-\lambda} m), ~\mbox{for}~ a,b\in\mathcal{A}, ~ m,n\in M.
		\end{align*}
		It is called the {\bf semi-direct product} of $\A$ and $M$, and denoted by $\A\ltimes_\rho M$.
	\end{prop}

	Let us recall the cohomology complex for a Lie
	conformal algebra $\A$ with coefficients in a  module $(M;\rho)$ (see \cite{DK3} for details). Set $C^{0}(\A,M)=M/\partial^M M$.   For $k\geq 1 $, denote by $C^k(\A,M)$ the space of $\C$-linear map $f:\A^{\otimes k}\rightarrow \C[\lambda_{1},\cdots,\lambda_{k-1}]\otimes M $ satisfying conformal sesquilinearity:
	\begin{align}
	\label{eq:coboundary1}f_{\lambda_1,\cdots,\lambda_{k-1}}(a_{1},\cdots,\partial a_{i},\cdots , a_{k})=&-\lambda_{i}f_{\lambda_1,\cdots,\lambda_{k-1}}(a_1,\cdots,a_k),\quad 1\leq i\leq k-1,\\
	\label{eq:coboundary2} f_{\lambda_1,\cdots,\lambda_{k-1}}(a_{1},\cdots , a_{k-1},\partial a_k)=&-\lambda^\dag_kf_{\lambda_1,\cdots,\lambda_{k-1}}(a_1,\cdots,a_k),
	\end{align}
	where $\lambda_{k}^{\dag}=-\sum_{j=1}^{k-1}\lambda_{j}-\partial^M$,  and  $f$ is skew-symmetric with respect to simultaneous permutations of the $a_i$'s and the $\la_i$'s in the sense that
	for every permutation $\sigma$ of the indices $\{1,\cdots,k\}$, that is, 	
	\begin{equation}
	\label{eq:coboundary3}	f_{\lambda_1,\cdots,\lambda_{k-1}}(a_{1},\cdots , a_{k-1},a_k)=(-1)^\sigma f_{\lambda_{\sigma(1)},\cdots,\lambda_{\sigma(k-1)}}(a_{\sigma(1)},\cdots , a_{\sigma(k-1)},a_{\sigma(k)})|_{\lambda_{k}\mapsto\lambda_{k}^{\dag}},
	\end{equation}
	where the notation in the RHS means that $\la_k$ is replaced by $\lambda_{k}^{\dag}=-\sum_{j=1}^{k-1}\lambda_{j}-\partial^M$, if it occurs.		
	
	For $ \bar{m}\in C^{0}(\A,M)=M/\partial^M M$, let $\mathbf{d}~ \bar{m}\in C^{1}(\A,M)$ be the following $\C[\partial]$-module homomorphism:
	\begin{align}\label{o-coboundary}
	(\mathbf{d}~ \bar{m})(a)=\rho(a)_{-\partial^M}m.
	\end{align}
	This is well defined since, if $\partial^M m\in \partial^M M$, the RHS is zero due to conformal sesquilinearity. For $f\in C^k(\A,M)$, with $k\geq 1$, define $\mathbf{d}f\in C^{k+1}(\A,M)$ by
	\begin{equation}\label{eq:coboundary LCA}
	\begin{split}
	&({\mathbf{d} f})_{\lambda_1,\cdots,\la_k}(a_1,a_2,\cdots,a_{k+1})=\sum_{i=1}^{k}(-1)^{i+1}\rho(a_i)_{\lambda_{i}} f_{\lambda_1,\cdots,\hat{\la_i},\cdots,\la_k}(a_1,\cdots,\hat{a}_i,\cdots,a_{k+1})\\& \quad\quad \quad\quad \quad\quad \quad \quad\quad \quad\quad +\sum_{i=1}^{k}(-1)^{i}f_{\lambda_1,\cdots,\hat{\la_i},\cdots,\la_k}(a_1,\cdots,\hat{a}_i,\cdots,a_k,[a_{i\lambda_{i}}a_{k+1}])\\
	&\quad\quad \quad\quad \quad\quad \quad\quad \quad \quad\quad +\sum_{i,j=1,i<j}^{k}(-1)^{k+i+j+1} f_{{\lambda_1,\cdots,\hat{\la}_i,\cdots,\hat{\la}_j,\cdots,\la_k,\la^\dag_{k+1}}}(a_1,\cdots,a_k,[a_{i\lambda_{i}}a_{j}])
	\\&\quad\quad \quad\quad \quad\quad \quad\quad \quad \quad\quad  +(-1)^k\rho(a_{k+1})_{\lambda^\dag_{k+1}} f_{\lambda_1,\cdots,\la_{k-1}}(a_1,\cdots,a_{k}),
	\end{split}
	\end{equation}
	where $a_{1},\cdots, a_{k+1}\in\A$, $\hat{a_i}$ means that the entry $a_i$ is omitted, and  $\lambda_{k+1}^{\dag}=-\sum_{j=1}^{k}\lambda_{j}-\partial^M$ with $\partial^M$ acting from the left.
	\begin{theo}{\rm(\cite{DK3})} For $f\in C^k(\A,M)$, we have $\mathbf{d}f\in C^{k+1}(\A,M)$ and $\mathbf{d}^2 f=0$. This makes $(C^*(\A,M), \mathbf{d})$ into a cohomology complex.
	\end{theo}				
	
	\section{Nijenhuis-Richardson bracket, quasi-twilled Lie conformal algebras and $L_\infty$-algebras} 
	In this section, we first  recall  the Nijenhuis-Richardson bracket for Lie conformal algebras and give the notion of bidegree on Lie conformal algebra cohomology complex. Then we introduce the notions of twilled Lie conformal algebras and quasi-twilled Lie conformal algebras and show that they induce a differential graded Lie algebra and an $L_\infty$-algebra,  respectively.

	\subsection{Nijenhuis-Richardson bracket for Lie conformal algebras and bidegrees}
	
	A permutation $\sigma\in\perm_n$ is called an {\bf $(i,n-i)$-unshuffle} if $\sigma(1)<\cdots<\sigma(i)$ and $\sigma(i+1)<\cdots<\sigma(n)$. If $i=0$ or $i=n$, we assume $\sigma=\Id$. The set of all $(i,n-i)$-unshuffles is denoted by $\perm_{(i,n-i)}$.
	
	Let $\A$ be a $\C[\partial]$-module. Set $C^*(\A,\A)=\oplus_{k\geq1}C^k(\A,\A)$, where $C^k(\A,\A)$ is the space of $\C$-linear maps from   $\A^{\otimes k}$ to $ \C[\lambda_{1},\cdots,\lambda_{k-1}]\otimes \A$ satisfying \eqref{eq:coboundary1}-\eqref{eq:coboundary3}.
	For $f\in C^m(\A,\A)$ and $g\in C^n(\A,\A)$, define the Nijenhuis-Richardson (NR) bracket on $C^*(\A,\A)$ by
	\begin{equation}
	[f,g]_\NR=f\diamond g-(-1)^{(m-1)(n-1)}g\diamond f,
	\end{equation}
	where $f\diamond g\in C^{m+n-1}(\A,\A)$ is defined by ($a_1,a_2,\cdots,a_{m+n-1}\in \A$)
	\begin{align*}
	&(f\diamond g)_{\lambda_1,\cdots,\lambda_{m+n-2}}(a_1,a_2,\cdots,a_{m+n-1})\nonumber\\
	\nonumber=&\sum_{\sigma\in \perm_{(n,m-1)}}(-1)^\sigma  f_{\lambda_{\sigma(1)}+\cdots+\lambda_{\sigma(n)},\lambda_{\sigma(n+1)},\cdots,\lambda_{\sigma(m+n-2)}}\big(g_{\lambda_{\sigma(1)},\cdots, \nonumber\lambda_{\sigma(n-1)}}(a_{\sigma(1)},\cdots,a_{\sigma(n)}\big),\\
	& a_{\sigma(n+1)},\cdots,a_{\sigma(m+n-1)})|_{\lambda_{m+n-1}\mapsto\lambda_{m+n-1}^{\dag}},
	\end{align*}
	where  $\lambda_{m+n-1}^{\dag}=-\sum_{i=1}^{m+n-2}\lambda_{i}-\partial$. Furthermore, we have
	\begin{lemm} {\rm (\cite{DK13,Wu})}
		$(C^*(\A,\A),[-,-]_\NR)$ is a graded Lie algebra. Moreover, a 2-cochain $\pi\in C^2(\A,\A)$ defines a Lie conformal algebra structure on $\A$ by
		$$[a_\lambda b]:=\pi_\lambda(a,b),\quad\forall~a,b\in\A,$$
		if and only if $[\pi,\pi]_\NR=0$.
	\end{lemm}
	
\emptycomment{	Let $(\A,\pi)$ be a Lie conformal algebra. For $k\geq 1$, define $\mathbf{d}_\pi:C^k(\A,\A)\rightarrow C^{k+1}(\A,\A)$ by
	$$\mathbf{d}_\pi f =(-1)^{k-1}[\pi,f]_\NR,\quad \forall~f\in C^k(\A,\A).$$
	More precisely, for $a_1,a_2,\cdots, a_{k+1}\in\A$, we have
	\begin{eqnarray*}
		&&(\mathbf{d}_\pi f)_{\lambda_1,\cdots,\lambda_k}(a_1,a_2,\cdots,a_{k+1})\\
		&=&\sum_{i=1}^{k+1}(-1)^{i+1}\pi_{\lambda_i} \big(a_{i},f_{\lambda_1,\cdots,\hat{\lambda_i},\cdots,\lambda_k}(a_{1},\cdots,\hat{a_i},\cdots ,a_{k+1})\big)_{\lambda_{k+1}\mapsto\lambda_{k+1}^{\dag}}\\
		&&+\sum_{1\leq i<j\leq k+1}(-1)^{i+j}f_{\lambda_i+\lambda_j,\lambda_1,\cdots,\hat{\lambda_i},\cdots,\hat{\lambda_j},\cdots,\lambda_k}
		\big(\pi_{\lambda_i}(a_{i},a_j),a_1,\cdots,\hat{a_i},\cdots,\hat{a_j},\cdots,a_{k+1}\big)|_{\lambda_{k+1}\mapsto\lambda_{k+1}^{\dag}}.
	\end{eqnarray*}
	Comparing this with \eqref{eq:coboundary LCA}, we obtain $\mathbf{d}_\pi(f)=\mathbf{d}(f)$ for $f\in C^k(\A,\A)$, $k\geq 1$.
	Because of the graded Jacobi identity, $\mathbf{d}_\pi$ is a square-zero derivation of degree $+1$ of $(C^*(\A,\A),[-,-]_\NR)$.}

	
	Let $\A_1$ and $\A_2$ be $\C[\partial]$-modules. In the following, the elements in $\A_1$ are usually denoted by $a,b, a_1,a_2,\cdots$ and the elements in $\A_2$ by $u,v,v_1,v_2,\cdots$. Let $f:\A_1^{\otimes k}\otimes \A_2^{\otimes l}\mapsto \C[\lambda_1,\cdots,\lambda_{k+l-1}]\otimes \A_1$ be a linear map satisfying \eqref{eq:coboundary1}, \eqref{eq:coboundary2} and the following condition:
	\begin{align}
	\label{eq:skew-condition}&f_{\lambda_1,\cdots,\lambda_{k+l-1}}(a_{1},\cdots ,a_k,v_{k+1},\cdots,v_{k+l})\nonumber \\
	=&(-1)^\sigma (-1)^\tau f_{\lambda_{\sigma(1)},\cdots,\lambda_{\sigma(k)},\lambda_{\tau(k+1)},\cdots,\lambda_{\tau(k+l-1)}}(a_{\sigma(1)},\cdots , a_{\sigma(k)},v_{\tau(k+1)},\cdots,v_{\tau(k+l)})|_{\lambda_{k+l}\mapsto\lambda_{k+l}^{\dag}},
	\end{align}
	for every permutation $\sigma$ of the indices $\{1,\cdots,k\}$ and  permutation $\tau$ of the indices $\{k+1,\cdots,k+l\}$.
	We can define a linear map $\hat{f}\in C^{k+l}(\A_1\oplus \A_2,\A_1\oplus \A_2)$ by
	\begin{eqnarray*}
		&&	\hat{f}_{\lambda_1,\cdots,\lambda_{k+l-1}}((a_1,v_1),\cdots,(a_{k+l},v_{k+l}))\\
		&=&\big(\sum_{\tau\in\perm_{(k,l)}}(-1)^\tau f_{\lambda_{\tau(1)},\cdots,\lambda_{\tau(k+l-1)}}(a_{\tau(1)},\cdots,a_{\tau(k)},v_{\tau(k+1)},\cdots,v_{\tau(k+l)}),0\big)|_{\lambda_{k+l}\mapsto\lambda_{k+l}^{\dag}}.
	\end{eqnarray*}
	Similarly, for a linear map $f:\A_1^{\otimes k}\otimes \A_2^{\otimes l}\mapsto \C[\lambda_1,\cdots,\lambda_{k+l-1}]\otimes \A_2$ satisfying \eqref{eq:coboundary1}, \eqref{eq:coboundary2} and \eqref{eq:skew-condition}, we obtain a linear map $\hat{f}\in C^{k+l}(\A_1\oplus \A_2,\A_1\oplus \A_2)$ by
	\begin{eqnarray*}
		&&	\hat{f}_{\lambda_1,\cdots,\lambda_{k+l-1}}((a_1,v_1),\cdots,(a_{k+l},v_{k+l}))\\ &=&\big(0,\sum_{\tau\in\perm_{(k,l)}}(-1)^\tau f_{\lambda_{\tau(1)},\cdots,\lambda_{\tau(k+l-1)}}(a_{\tau(1)},\cdots,a_{\tau(k)},v_{\tau(k+1)},\cdots,v_{\tau(k+l)})\big)|_{\lambda_{k+l}\mapsto\lambda_{k+l}^{\dag}}.
	\end{eqnarray*}
	The linear map $\hat{f}$ is called a {\bf lift} of $f$. For
	example, the lifts of linear maps $\alpha:\A_1\otimes\A_1\rightarrow\A_1[\lambda]$
	and $\beta:\A_1\otimes\A_2\rightarrow\A_2[\lambda]$ are  respectively
	given by
	\begin{align}
	\label{semi-direct-1}\hat{\alpha}_\lambda\big((a_1,v_1),(a_2,v_2)\big)=&(\alpha_\lambda(a_1,a_2),0),\\
	\label{semi-direct-2}\hat{\beta}_\lambda\big((a_1,v_1),(a_2,v_2)\big)=&(0,\beta_\lambda(a_1,v_2)-\beta_{-\partial-\lambda}(a_2,v_1)).
	\end{align}
	
	\begin{defi} Let $\A_1$ and $\A_2$ be $\C[\partial]$-modules.  A cochain
		$f\in C^{k+l+1}(\A_1\oplus \A_2,\A_1\oplus \A_2)$ has a
		{\bf bidegree} $k|l$, if the following conditions hold:
		\begin{itemize}
			\item[\rm(i)] If $X$ is an element in $\A_1^{\otimes k+1}\otimes \A_2^{\otimes l}$, then $f(X)\in\C[\lambda_1,\cdots,\lambda_{k+l}]\otimes\A_1;$
			\item[\rm(ii)] If $X$ is an element in $\A_1^{\otimes k}\otimes \A_2^{\otimes l+1}$, then $f(X)\in\C[\lambda_1,\cdots,\lambda_{k+l}]\otimes\A_2;$
			\item[\rm(iii)] All the other cases, $f(X)=0.$
		\end{itemize}
	\end{defi}
	We call $f\in C^{n}(\A_1\oplus \A_2,\A_1\oplus \A_2)$ a {\bf homogeneous cochain} if $f$ has a bidegree and denote its bidegree by $||f||$.
	Obviously, $\hat{\alpha}$ and $\hat{\beta}$ given by \eqref{semi-direct-1} and \eqref{semi-direct-2} are elements in $C^2(\A_1\oplus\A_2,\A_1\oplus\A_2)$  with $||\hat{\alpha}||=||\hat{\beta}||=1|0$. Naturally we obtain a homogeneous linear map of bidegree $1|0$, namely, $\hat{\mu}:=\hat{\alpha}+\hat{\beta}.$
	Observe that $\hat{\mu}$ is a multiplication of the semi-direct product type:
	$$
	\hat{\mu}_\lambda\big((a_1,v_1),(a_2,v_2)\big)=(\alpha_\lambda(a_1,a_2),\beta_\lambda(a_1,v_2)-\beta_{-\lambda-\partial}(a_2,v_1)),~a_1,a_2\in\A_1,~v_1,v_2\in\A_2.
	$$
	Even though $\hat{\mu}$ is not a lift (there is no $\mu$), we still use the symbol for our convenience below.
	
	The following lemma shows that the NR bracket on $C^*(\A_1\oplus\A_2,\A_1\oplus\A_2)$ is compatible with the
	bigrading.
	\begin{lemm}\label{lem:bidegree perserve}
		If $||f||=k_f|l_f$ and $||g||=k_g|l_g$, then $[f,g]_{\NR}$ has the bidegree $k_f+k_g|l_f+l_g.$
	\end{lemm}
	
	It is straightforward to check that
	\begin{lemm}\label{lem:zero condition}
		If $||f||=-1|l$ (resp. $l|-1$) and $||g||=-1|k$ (resp. $k|-1$), then $[f,g]_{\NR}=0.$
	\end{lemm}
	
	\begin{prop}\label{pro:MC}
		Let $(\A,\pi)$ be a Lie conformal algebra, $M$ a $\C[\partial]$-module and  $\rho:\A\rightarrow {\rm Cend}(M)$ a linear map satisfying $\rho(\partial(a))_\lambda=-\lambda\rho(a)_\lambda$. Then $(M;\rho)$ is  a module of $\A$ if and only if
		\begin{equation*}
		[\hat{\pi}+\hat{\rho},\hat{\pi}+\hat{\rho}]_{\NR}=0.
		\end{equation*}
	\end{prop}
	\begin{proof}
		It follows by a direct calculation.
	\end{proof}
	
	Let $(\A,\pi)$ be a Lie conformal algebra and $(M;\rho)$  a module
	over $\A$. By the definitions of lift and bidegree, $\hat{\pi}+\hat{\rho}\in C^{1\mid 0}(\A\oplus
	M,\A\oplus M)$, and the subspace $C^k(\A,M)$ is identified with $C^{k\mid -1}(\A\oplus
	M,\A\oplus M)$. Set $C^*(\A,M)=\oplus_{k=1}^{+\infty}C^k(\A,M)$.
	Define the coboundary operator $\dM_{\pi+\rho}:C^k(\A,M)\rightarrow C^{k+1}(\A,M)$ by
	\begin{equation}\label{eq:CE-operator}
	\dM_{\pi+\rho} f:=(-1)^{k-1}[\hat{\pi}+\hat{\rho},\hat{f}]_{\NR},\quad \forall~f\in C^k(\A,M).
	\end{equation}
	By Lemma \ref{lem:bidegree perserve}, $\dM_{\pi+\rho} f\in C^{k+1}(\A,M) $. By Proposition \ref{pro:MC} and the graded Jacobi identity, we have $\dM_{\pi+\rho}\circ \dM_{\pi+\rho}=0$. Thus we obtain a well-defined cochain complex $(C^*(\A,M),\dM_{\pi+\rho})$.
	
	By a direct calculation, we have
	\begin{prop}\label{pro:CE-operator} Let $(\A,\pi)$ be a Lie conformal algebra and $(M;\rho)$  a module over
		$\A$. Then for all $f\in C^k(\A,M)$ and
		$a_1,a_2,\cdots,a_{k+1}\in \A$, we have
		\begin{eqnarray*}
			&&	(\dM_{\pi+\rho} f)_{\lambda_1,\cdots,\lambda_k}(a_1,\ldots,a_{k+1})\\
			&&=\sum_{i=1}^{k+1}(-1)^{i+1}\rho({a_i})_{\lambda_i} f_{\lambda_1,\cdots,\hat{\lambda_i},\cdots,\lambda_k}(a_1,\ldots,\hat{a_i},\ldots,a_{k+1})|_{\lambda_{k+1}\mapsto\lambda_{k+1}^{\dag}}\\
			&&+\sum_{1\leq i<j\leq k+1}(-1)^{i+j}f_{\lambda_i+\lambda_j,\lambda_1,\cdots,\hat{\lambda_i},\cdots,\hat{\lambda_j},\cdots,\lambda_k}(\pi_{\lambda_i}(a_i,a_j),\ldots,\hat{a_{i}},\ldots,\hat{a_{j}},\ldots,a_{k+1})|_{\lambda_{k+1}\mapsto\lambda_{k+1}^{\dag}}.
		\end{eqnarray*}
		Thus the coboundary operator $\dM_{\pi+\rho}:C^k(\A,M)\rightarrow C^{k+1}(\A,M)$ defined by \eqref{eq:CE-operator} is exactly the Lie conformal algebra coboundary operator for $(\A, \pi)$ with coefficients in $(M;\rho)$  for $k\geq 1$.
	\end{prop}
	
	\subsection{Quasi-twilled Lie  conformal algebras and $L_\infty$-algebras}
	
	Let $(\A,[\cdot_\lambda\cdot])$ be a Lie conformal algebra with a decomposition into the direct sum of  two $\C[\partial]$-modules $\A_1$ and $\A_2$, that  is, $\A=\A_1\oplus \A_2.$
	
	\begin{lemm}\label{lem:decomposition}
		Let $\pi\in C^2(\A,\A)$	be a  $2$-cochain. Then $\pi$ can be uniquely decomposed into four homogeneous linear maps
		$$
		\pi=\hat{\phi}_1+\hat{\mu}_1+\hat{\mu}_2+\hat{\phi}_2,
		$$
		where the bidegrees of $\phi_1$, $\mu_1$, $\mu_2$ and $\phi_2$ are  $2|-1,~1|0,~0|1$ and $-1|2$, respectively.
	\end{lemm}
	\begin{proof}
		By the definition of bidegree,  the space $C^2(\A,\A)$ can be decomposed into four subspaces
		$$
		C^2(\A,\A)=	C^{2|-1}(\A,\A)+C^{1|0}(\A,\A)+C^{0|1}(\A,\A)+C^{-1|2}(\A,\A).
		$$
		Thus $\pi$ is uniquely decomposed into homogeneous linear maps of bidegrees $2|-1,~1|0,~0|1$ and $-1|2$.
	\end{proof}
	
	Let ${\rm p}_1:\A\rightarrow\A_1$ and ${\rm p}_2: \A\rightarrow\A_2$ be the natural projections. For $a,b\in\A_1$, $u,v\in\A_2$, define
	\begin{align*}
	&[a_\lambda b]_1={\rm p_1}([a_\lambda b]),\quad \rho_{2 }(v)_\lambda a={\rm p_1}([v_\lambda a]),\quad \phi_{2\lambda}(u,v)={\rm p_1}([u_\lambda v]),\\
	&{[u_\lambda v]}_2={\rm p_2}([u_\lambda v]),\quad \rho_{1 }(a)_\lambda v={\rm p_2}([a_\lambda v]),\quad \phi_{1\lambda}(a,b)={\rm p_2}([a_\lambda b]).
	\end{align*}
	Then the $\lambda$-bracket of $\A$  can be uniquely written as
	\begin{equation*}
	[(a,u)_\lambda (b,v)]=([a_\lambda b]_1+\rho_{2}(u)_\lambda b-\rho_{2}(v)_{-\lambda-\partial}a+\phi_{2\lambda}(u,v),[u_\lambda v]_2+\rho_{1}(a)_\lambda v-\rho_{1}(b)_{-\lambda-\partial}u+\phi_{1\lambda}(a,b)).
	\end{equation*}
	
	Now we denote the Lie conformal algebra structure on $\A$ by $\Pi$, i.e. $$\Pi_\lambda((a,u),(b,v)):=[(a,u)_\lambda (b,v)].$$ Set $\Pi_\lambda=\hat{\phi}_1+\hat{\mu}_1+\hat{\mu}_2+\hat{\phi}_2$ as in Lemma \ref{lem:decomposition}.  Then
	\begin{align}
	\label{bracket-1}\hat{\phi}_{1\lambda}((a,u),(b,v))&=(0,\phi_{1\lambda}(a,b)),\\
	\label{bracket-2}\hat{\mu}_{1\lambda}((a,u),(b,v))&=([a_\lambda b]_1,\rho_{1}(a)_\lambda v-\rho_{1}(b)_{-\lambda-\partial}u),\\
	\label{bracket-3}\hat{\mu}_{2\lambda}((a,u),(b,v))&=(\rho_{2}(u)_\lambda b-\rho_{2}(v)_{-\lambda-\partial}a,[u_\lambda v]_2),\\
	\label{bracket-4}\hat{\phi}_{2\lambda}((a,u),(b,v))&=(\phi_{2\,\lambda}(u,v),0).
	\end{align}
	
	\begin{lemm}\label{proto-twilled}
		The Maurer-Cartan equation $[\Pi,\Pi]_{\NR}=0$ is equivalent to the following  conditions:
		\begin{eqnarray}\label{eq:equivalent of pi}
		\left\{\begin{array}{rcl}
		{}[\hat{\mu}_1,\hat{\phi}_1]_{\NR}&=&0,\\
		{}\frac{1}{2}[\hat{\mu}_1,\hat{\mu}_1]_{\NR}+[\hat{\mu}_2,\hat{\phi}_1]_{\NR}&=&0,\\
		{}[\hat{\mu}_1,\hat{\mu}_2]_{\NR}+[\hat{\phi}_1,\hat{\phi}_2]_{\NR}&=&0,\\
		{}\frac{1}{2}[\hat{\mu}_2,\hat{\mu}_2]_{\NR}+[\hat{\mu}_1,\hat{\phi}_2]_{\NR}&=&0,\\
		{}[\hat{\mu}_2,\hat{\phi}_2]_{\NR}&=&0.
		\end{array}\right.
		\end{eqnarray}
	\end{lemm}
	\begin{proof}
		By Lemma \ref{lem:zero condition}, we have
		\begin{align*}
		[\Pi,\Pi]_{\NR}=&[\hat{\phi}_1+\hat{\mu}_1+\hat{\mu}_2+\hat{\phi}_2,\hat{\phi}_1+\hat{\mu}_1+\hat{\mu}_2+\hat{\phi}_2]_{\NR}\\
		=&[\hat{\phi}_1,\hat{\mu}_1]_{\NR}+[\hat{\phi}_1,\hat{\mu}_2]_{\NR}+[\hat{\phi}_1,\hat{\phi}_2]_{\NR}+[\hat{\mu}_1,\hat{\phi}_1]_{\NR}+[\hat{\mu}_1,\hat{\mu}_1]_{\NR}
		\\&+[\hat{\mu}_1,\hat{\mu}_2]_{\NR}+[\hat{\mu}_1,\hat{\phi}_2]_{\NR}+[\hat{\mu}_2,\hat{\phi}_1]_{\NR}+[\hat{\mu}_2,\hat{\mu}_1]_{\NR}\\&+[\hat{\mu}_2,\hat{\mu}_2]_{\NR} +[\hat{\mu}_2,\hat{\phi}_2]_{\NR}+[\hat{\phi}_2,\hat{\phi}_1]_{\NR}+[\hat{\phi}_2,\hat{\mu}_1]_{\NR}+[\hat{\phi}_2,\hat{\mu}_2]_{\NR}\\
		=&(2[\hat{\mu}_1,\hat{\phi}_1]_{\NR})+([\hat{\mu}_1,\hat{\mu}_1]_{\NR}+2[\hat{\mu}_2,\hat{\phi}_1]_{\NR})+(2[\hat{\mu}_1,\hat{\mu}_2]_{\NR}+2[\hat{\phi}_1,\hat{\phi}_2]_{\NR})\\
		&+([\hat{\mu}_2,\hat{\mu}_2]_{\NR}+2[\hat{\mu}_1,\hat{\phi}_2]_{\NR})+(2[\hat{\mu}_2,\hat{\phi}_2]_{\NR}).
		\end{align*}
		By Lemma \ref{lem:bidegree perserve} and the definition of bidegree,  $[\Pi,\Pi]_{\NR}=0$ if and only if  \eqref{eq:equivalent of pi} holds.
	\end{proof}
	\begin{defi}{\rm
			Let $\A$ be a Lie conformal algebra with a decomposition into the direct sum of  two $\C[\partial]$-modules $\A_1$ and $\A_2$ and the Lie conformal algebra structure
			$\Pi=\hat{\phi}_1+\hat{\mu}_1+\hat{\mu}_2+\hat{\phi}_2$.
			\begin{itemize}
				\item[(i)]The triple $(\A,\A_1,\A_2)$ is called a {\bf twilled Lie conformal algebra} if $\phi_1=\phi_2=0$, or equivalently, $\A_1$ and $\A_2$ are subalgebras of $\A$. In this case, we denote $(\A,\A_1,\A_2)$ by $\A=\A_1\bowtie\A_2$.
				\item[(ii)]The triple $(\A,\A_1,\A_2)$ is called a {\bf quasi-twilled Lie conformal algebra} if $\phi_2=0$, i.e.,  $\A_2$ is a subalgebra of $\A$.
			\end{itemize}
		}\end{defi}
		
		By  Lemma \ref{proto-twilled}, we have
		\begin{lemm}\label{lem:twillL}
			The triple $(\A,\A_1,\A_2)$ is a twilled Lie conformal algebra if and only if the following  conditions hold:
			\begin{eqnarray}
			\label{twilled-1}\frac{1}{2}[\hat{\mu}_1,\hat{\mu}_1]_{\NR}&=&0,\\
			\label{twilled-2}[\hat{\mu}_1,\hat{\mu}_2]_{\NR}&=&0,\\
			\label{twilled-3}\frac{1}{2}[\hat{\mu}_2,\hat{\mu}_2]_{\NR}&=&0.
			\end{eqnarray}
		\end{lemm}
		By \eqref{twilled-1},  $\hat{\mu}_1$ is a  Lie conformal algebra structure on $\A=\A_1\oplus\A_2$. By \eqref{bracket-2}, $(\A_1,[\cdot_\lambda\cdot]_{1})$ is a Lie conformal algebra and $(\A_2;\rho_{1})$ is a module over $(\A_1,[\cdot_\lambda\cdot]_{1})$. Similarly,$(\A_2,[\cdot_\lambda\cdot]_{2})$ is a Lie conformal algebra and $(\A_1;\rho_{2})$ is a module over $(\A_2,[\cdot_\lambda\cdot]_{2})$. Hence the $\lambda$-bracket on  the twilled Lie conformal algebra $\A=\A_1\bowtie\A_2$ is given by ($a,b\in\A_1,~u,v\in\A_2$):
		\begin{equation}\label{eq:twilled bracket}
		[(a,u)_\lambda (b,v)]=([a_\lambda b]_1+\rho_{2}(u)_\lambda b-\rho_{2}(v)_{-\lambda-\partial}a,[u_\lambda v]_2+\rho_{1}(a)_\lambda v-\rho_{1}(b)_{-\lambda-\partial}u).
		\end{equation}

		\begin{lemm}\label{lem:quasi-twillL}
			$(\A,\A_1,\A_2)$ forms a quasi-twilled Lie conformal algebra if and only if there hold
			\begin{eqnarray}
			\label{quasi-1}[\hat{\mu}_1,\hat{\phi}_1]_{\NR}&=&0,\\
			\label{quasi-2}\frac{1}{2}[\hat{\mu}_1,\hat{\mu}_1]_{\NR}+[\hat{\mu}_2,\hat{\phi}_1]_{\NR}&=&0,\\
			\label{quasi-3}[\hat{\mu}_1,\hat{\mu}_2]_{\NR}&=&0,\\
			\label{quasi-4}\frac{1}{2}[\hat{\mu}_2,\hat{\mu}_2]_{\NR}&=&0.
			\end{eqnarray}
		\end{lemm}
		The $\lambda$-bracket on  the quasi-twilled Lie conformal algebra $(\A,\A_1,\A_2)$ is given by
		\begin{equation}\label{eq:quasi-twilled bracket}
		[(a,u)_\lambda (b,v)]=([a_\lambda b]_1+\rho_{2}(u)_\lambda b-\rho_{2}(v)_{-\lambda-\partial}a,[u_\lambda v]_2+\rho_{1}(a)_\lambda v-\rho_{1}(b)_{-\lambda-\partial}u+\phi_{1\lambda}(a,b)).
		\end{equation}
		Since $[\hat{\mu}_1,\hat{\mu}_1]_{\NR}$ is not zero in general, $(\A_1,[\cdot_\lambda\cdot]_1)$ is not a Lie conformal algebra. By \eqref{quasi-4}, $(\A_2,[\cdot_\lambda\cdot]_{2})$ is a Lie conformal algebra and $(\A_1;\rho_{2})$ is a module over $(\A_2,[\cdot_\lambda\cdot]_{2})$.

		\begin{coro} \label{pro:twisted semi-direct product}
			Let $(\A,[\cdot_\lambda\cdot])$ be a Lie conformal algebra and $(M;\rho)$  a module over
			$\A$. For any $2$-cocycle $\phi$ in $C^2(\A, M)$,  the $\C[\p]$-module $\A\oplus M$ carries a quasi-twilled Lie conformal algebra structure via
			\begin{align}\label{twist-sum}
			[(a,m)_\lambda(b,n)]^\phi=\big([a_\lambda b], \rho(a)_\lambda n-\rho(b)_{-\lambda-\p} m+\phi_\lambda(a,b)\big),
			\end{align}
			where $a,b\in\A$ and $m,n\in M.$ We called it the {\bf $\phi$-twisted semi-direct product of $\A$ and $M$}, denoted by $\A\ltimes_{\phi}M$.
		\end{coro}
		
		
		
		\begin{defi}{\rm (\cite{LS,LM})
				An {\bf $L_\infty$-algebra} is a graded vector space $\g=\bigoplus_{i\in\mathbb{N}}\g^i$ equipped with a collection  of multilinear maps $l_k:{\otimes^k}\g\rightarrow\g$ of degree $2-k$, satisfying
				\begin{itemize}	
					\item[(i)]Skew-symmetry: $l_k\big(v_{\sigma(1)}, \cdots, v_{\sigma(k)}\big)=\chi(\sigma)l_k(v_1, \cdots, v_k),$ for $k\ge 1$ and $\sigma\in S_k$. Here $\chi(\sigma)$ is the Koszul sign.
					\item[(ii)] Higher Jacobi identity: for $n\ge 1$,
					\begin{eqnarray*}\label{sh-Lie}
					\sum_{i+j=n+1}(-1)^{i}\sum_{{\sigma\in \perm_{(i,n-i)}} }\chi(\sigma)l_j(l_i(v_{\sigma(1)},\cdots,v_{\sigma(i)}),v_{\sigma(i+1)},\cdots,v_{\sigma(n)})=0,
					\end{eqnarray*}
					where $v_1,\cdots,v_n$ are homogeneous elements in $\g$.
				\end{itemize}}
			\end{defi}
			Let $(\g=\bigoplus_{i\in\mathbb{N}}\g^i,\{l_k\}_{k=1}^\infty)$ be an $L_\infty$-algebra. An element $\alpha\in \g^1$ is called an {\bf Maurer-Cartan element} if it satisfies the following {\bf Maurer-Cartan equation}:
			\begin{eqnarray*}\label{MC-equation}
			\sum_{k=1}^{+\infty}\frac{1}{k!}l_k(\alpha,\cdots,\alpha)=0.
			\end{eqnarray*}
			
			The following theorem says that a quasi-twilled Lie conformal algebra induces an $L_\infty$-algebra.
			\begin{theo}\label{quasi-as-shLie}
				Let $(\A,\A_1,\A_2, \Pi)$ be a quasi-twilled Lie conformal algebra with 
				$\Pi=\hat{\phi}_1+\hat{\mu}_1+\hat{\mu}_2$.  Define $d_{\hat{\mu}_2}:C^m(\A_2,\A_1)\rightarrow C^{m+1}(\A_2,\A_1)$, $[\cdot,\cdot]_{{\hat{\mu}_1}}:C^m(\A_2,\A_1)\times C^n(\A_2,\A_1)\rightarrow C^{m+n}(\A_2,\A_1)$ and $[\cdot,\cdot,\cdot]_{\hat{\phi}_1}:C^m(\A_2,\A_1)\times C^n(\A_2,\A_1)\times C^k(\A_2,\A_1)\rightarrow C^{m+n+k-1}(\A_2,\A_1)$ as follows:
				\begin{align}
				\label{eq:shLie1}d_{\hat{\mu}_2}(f_1)=&[\hat{\mu}_2,\hat{f}_1]_{\NR},\\
				\label{eq:shLie2}{[f_1,f_2]_{\hat{\mu}_1}}=&(-1)^{m-1}[[\hat{\mu}_1,\hat{f}_1]_{\NR},\hat{f}_2]_{\NR},\\		\label{eq:shLie3}{[f_1,f_2,f_3]}_{\hat{\phi}_1}=&(-1)^{n-1}[[[\hat{\phi}_1,\hat{f}_1]_{\NR},\hat{f}_2]_{\NR},\hat{f}_3]_{\NR},
				\end{align}
				for $f_1\in C^m(\A_2,\A_1),~f_2\in C^n(\A_2,\A_1)$, $f_3\in C^k(\A_2,\A_1).$ Then $(C^*(\A_2,\A_1),d_{\hat{\mu}_2},[\cdot,\cdot]_{{\hat{\mu}_1}},[\cdot,\cdot,\cdot]_{\hat{\phi}_1})$ is an $L_\infty$-algebra.
			\end{theo}
			\begin{proof}
				Set $d_0:=[\hat{\mu}_2,\cdot]_{\NR}$. By the graded Jacobi identity of $[\cdot,\cdot]_{\NR}$, $(C^*(\A,\A),[\cdot,\cdot]_{\NR},d_0)$ is a differential graded Lie algebra.  By Lemma \ref{lem:bidegree perserve}, the brackets on $sC^*(\A,\A)$ given by
				\begin{align*}
				d_{\hat{\mu}_2}(sf_1)=&s[\hat{\mu}_2,f_1]_{\NR},\\
				{[sf_1,sf_2]_{\hat{\mu}_1}}=&(-1)^{|f_1|}s[[\hat{\mu}_1,f_1]_{\NR},f_2]_{\NR},\\
				{[sf_1,sf_2,sf_3]}_{\hat{\phi}_1}=&(-1)^{|f_2|}s[[[\hat{\phi}_1,f_1]_{\NR},f_2]_{\NR},f_3]_{\NR},\\
				l_i=&0,\,\,\,\,i\ge4,
				\end{align*}
				are closed on $C^*(\A_2,\A_1)$, where $f_1,~f_2,~f_3\in C^*(\A_2,\A_1)$ and  $s:C^*(\A_2,\A_1)\rightarrow sC^*(\A_2,\A_1)$ is the  suspension operator by assigning  $C^*(\A_2,\A_1)$ to the graded vector space $sC^*(\A_2,\A_1)$ with $|sC^i(\A_2,\A_1)|:=i-1$. Thus $C^*(\A_2,\A_1)$ is an abelian subalgebra of the graded Lie algebra $(C^*(\A,\A),[\cdot,\cdot]_{\NR})$. The rest follows from \cite[Corollary 3.5]{Uch1}.
			\end{proof}
			
			\begin{coro}\label{pro:twilled GLA}
				Let $\A_1\bowtie\A_2$ be a twilled Lie conformal algebra with the Lie conformal algebra structure
				$\Pi=\hat{\mu}_1+\hat{\mu}_2$.. Then $(C^*(\A_2,\A_1),d_{\hat{\mu}_2},[\cdot,\cdot]_{{\hat{\mu}_1}})$ is a differential graded Lie algebra, where $d_{\hat{\mu}_2}$	and $[\cdot,\cdot]_{{\hat{\mu}_1}} $ are given by \eqref{eq:shLie1} and \eqref{eq:shLie2}, respectively.
			\end{coro}

			\section{Twisting on Lie conformal algebras and  relative Rota-Baxter type operators}
			In this section, we study twisting theory of Lie conformal algebras and characterize relative Rota-Baxter type operators as  Maurer-Cartan elements of a suitable $L_\infty$-algebra.
			
			Let $(\A,\Pi) $ be a Lie conformal algebra with a decomposition into the direct sum of  two $\C[\partial]$-modules $\A_1$ and $\A_2$ and 
			$\Pi=\hat{\phi}_1+\hat{\mu}_1+\hat{\mu}_2+\hat{\phi}_2$.
			Let $\hat{H}$ be the lift of a $\C[\partial]$-module homomorphism $H:\A_2\rightarrow\A_1$. Notice that the bidegree of $H$ is $-1|1$. For the operator $X_{\hat{H}}:=[\cdot,\hat{H}]_{\NR}$,  we define
			$$e^{X_{\hat{H}}}(\cdot):={\Id}+X_{\hat{H}}+\frac{1}{2!}X^2_{\hat{H}}+\frac{1}{3!}X^3_{\hat{H}}+\cdots,$$
			where $X^2_{\hat{H}}:=[[\cdot,\hat{H}]_{\NR},\hat{H}]_{\NR}$ and $X_{\hat{H}}^n$ for $n\geq 3$ is defined similarly. As $\hat{H}\circ \hat{H}=0$, the operator $e^{X_{\hat{H}}}$ is well-defined.
			
			\begin{defi}{\rm
					The transformation $\Pi^{H}:=e^{X_{\hat{H}}}(\Pi)$ is called a {\bf twisting} of $\Pi$ by $H$.}
			\end{defi}
			
			The following lemma is useful.
			\begin{lemm}\label{lem:twist}
				\begin{itemize}
					\item[{\rm (i)}] $\Pi^{H}=\Pi+[\Pi,\hat{H}]_{\NR}+\half[[\Pi,\hat{H}]_{\NR},\hat{H}]_{\NR}+\frac{1}{6}[[[\Pi,\hat{H}]_{\NR},\hat{H}]_{\NR},\hat{H}]_{\NR}$;
					\item[{\rm (ii)}] 	$\Pi_\lambda^{H}=e^{-\hat{H}}\circ \Pi_\lambda\circ (e^{\hat{H}}\otimes e^{\hat{H}})$.
				\end{itemize}
			\end{lemm}
			\begin{proof}
				(i) Since $||H||=-1|1$, $X^i_{\hat{H}}(\Pi)=0$ for $i\geq 4$ by Lemmas \ref{lem:bidegree perserve} and \ref{lem:zero condition}. This proves (i).
				
				(ii) For $a_1,a_2\in\A_1,v_1,v_2\in \A_2$, we compute separately
				\begin{align*}
				&([\Pi,\hat{H}]_{\NR})_\lambda\big((a_1,v_1),(a_2,v_2)\big)\\
				=&\Pi_\lambda((H(v_1),0),(a_2,v_2))+\Pi_\lambda((a_1,v_1),(H(v_2),0))-\hat{H}(\Pi_\lambda((a_1,v_1),(a_2,v_2)));\\
				&	([[\Pi,\hat{H}]_{\NR},\hat{H}]_{\NR})_\lambda\big((a_1,v_1),(a_2,v_2)\big)\\
				=&2\Pi_\lambda((H(v_1),0),(H(v_2),0))-2\hat{H}\Pi_\lambda((H(v_1),0),(a_2,v_2))-2\hat{H}\Pi_\lambda((a_1,v_1),(H(v_2),0));\\	
				&([[[\pi,\hat{H}]_{\NR},\hat{H}]_{\NR},\hat{H}]_{\NR})_\lambda\big((a_1,v_1),(a_2,v_2)\big)
				=-6\hat{H}\Pi_\lambda((H(v_1),0),(H(v_2),0)).
				\end{align*}
				By assertion (i), we have
				\begin{align*} \Pi_\lambda^{H}=&\Pi_\lambda-\hat{H}\circ\Pi_\lambda+\Pi_\lambda\circ(\hat{H}\otimes{\Id})+\Pi_\lambda\circ({\Id}\otimes\hat{H})-\hat{H}\circ\Pi_\lambda\circ({\Id}\otimes\hat{H})\\	\nonumber&-\hat{H}\circ\Pi_\lambda\circ(\hat{H}\otimes{\Id})+\Pi_\lambda\circ(\hat{H}\otimes\hat{H})-\hat{H}\circ\Pi_\lambda\circ(\hat{H}\otimes\hat{H}).
				\end{align*}
				On the other hand, by $\hat{H}\circ\hat{H}=0$, we have
				\begin{align*}
				e^{-\hat{H}}\circ \Pi_\lambda \circ (e^{\hat{H}}\otimes e^{\hat{H}})=&({\Id}-\hat{H})\circ\Pi_\lambda\circ (({\Id}+\hat{H})\otimes({\Id}+\hat{H}))\\
				=&\Pi_\lambda+\Pi_\lambda\circ({\Id}\otimes\hat{H})+\Pi_\lambda\circ(\hat{H}\otimes{\Id})+\Pi_\lambda\circ(\hat{H}\otimes\hat{H})-\hat{H}\circ\Pi_\lambda\\	&-\hat{H}\circ\Pi_\lambda\circ({\Id}\otimes\hat{H})-\hat{H}\circ\Pi_\lambda\circ(\hat{H}\otimes{\Id})-\hat{H}\circ\Pi_\lambda\circ(\hat{H}\otimes\hat{H}).
				\end{align*}
				This proves assertion (ii).
			\end{proof}

			\begin{prop}\label{pro:main twisting prop1}
				The twisting $\Pi^{H}$ of $\Pi$ is a  Lie conformal algebra structure on $\A$, namely, $[\Pi^{H},\Pi^{H}]_{\NR}=0.$ Moreover,
				$e^{\hat{H}}:(\A,\Pi^{H})\rightarrow(\A,\Pi)$ is a Lie conformal algebra isomorphism.
			\end{prop}
			\begin{proof} By Lemma \ref{lem:twist}, we have
				\begin{align*}	([\Pi^{H},\Pi^{H}]_{\NR})_{\lambda_1,\lambda_2}=2(\Pi^{H}\diamond\Pi^{H})_{\lambda_1,\lambda_2}=&2e^{-\hat{H}}\circ(\Pi\diamond\Pi)_{\lambda_1,\lambda_2}\circ(e^{\hat{H}}\otimes e^{\hat{H}}\otimes e^{\hat{H}})\\
				=&e^{-\hat{H}}\circ([\Pi,\Pi]_{\NR})_{\lambda_1,\lambda_2}\circ(e^{\hat{H}}\otimes e^{\hat{H}}\otimes e^{\hat{H}})=0.
				\end{align*}
				The second claim follows directly.
			\end{proof}

			Since $\Pi^{H}$ is a $2$-cochain, $\Pi^H$ is also decomposed into the unique four substructures with respect to the bidegrees. The relations between $\Pi^H$ and $\Pi$ are given by the following theorem.
			\begin{theo}\label{thm:twist}
				Assume that $\Pi=\hat{\phi}_1+\hat{\mu}_1+\hat{\mu}_2+\hat{\phi}_2$. Then $\Pi^{H}=\hat{\phi}_1^{H}+\hat{\mu}_1^{H}+\hat{\mu}_2^{H}+\hat{\phi}_2^{H}$, where
				\begin{align}
				\label{twisting-1}\hat{\phi}_1^{H}=&\hat{\phi}_1,\\
				\label{twisting-2}\hat{\mu}_1^{H}=&\hat{\mu}_1+[\hat{\phi}_1,\hat{H}]_{\NR},\\
				\label{twisting-3}\hat{\mu}_2^{H}=&\hat{\mu}_2+d_{\hat{\mu}_1}\hat{H}+\half[[\hat{\phi}_1,\hat{H}]_{\NR},\hat{H}]_{\NR},\\ \label{twisting-4}\hat{\phi}_2^{H}=&\hat{\phi}_2+d_{\hat{\mu}_2}\hat{H}+\half[\hat{H},\hat{H}]_{\hat{\mu}_1}+\frac{1}{6}[[[\hat{\phi}_1,\hat{H}]_{\NR},\hat{H}]_{\NR},\hat{H}]_{\NR},
				\end{align}
				where $d_{\hat{\mu}_i}:=[\hat{\mu}_i,-]_{\NR}~(i=1,2)$ and $[\hat{H},\hat{H}]_{\hat{\mu}_1}:=[[\hat{\mu}_1,\hat{H}]_{\NR},\hat{H}]_{\NR}$.
			\end{theo}
			\begin{proof}
				By the decomposition of bidegree and Lemmas \ref{lem:bidegree perserve} and \ref{lem:zero condition}, the theorem follows.
			\end{proof}
			\subsection{The case of twilled Lie conformal  algebras}
			Let $(\A_1\bowtie\A_2,\Pi)$ be a twilled Lie conformal algebra with
			$\Pi=\hat{\mu}_1+\hat{\mu}_2$. The twisted structure $\Pi^H=\hat{\mu}_1^H+\hat{\mu}_2^{H}+\hat{\phi}_2^{H}$ by $H:\A_2\rightarrow \A_1$ is given by
			\begin{eqnarray*}
			\label{twilled-twisting-1}\hat{\mu}_1^{H}&=&\hat{\mu}_1,\\
			\label{twilled-twisting-2}\hat{\mu}_2^{H}&=&\hat{\mu}_2+d_{\hat{\mu}_1}\hat{H},\\
			\label{twilled-twisting-3}\hat{\phi}_2^{H}&=&d_{\hat{\mu}_2}\hat{H}+\half[\hat{H},\hat{H}]_{\hat{\mu}_1}.
			\end{eqnarray*}
			We have shown in Corollary \ref{pro:twilled GLA} that there is a differential graded Lie algebra structure  associated to $(\A_1\bowtie\A_2,\Pi)$. Now we show that the Maurer-Cartan element of this differential graded Lie algebra can give a new twilled Lie conformal algebra by the twisting transformation.
			\begin{prop}\label{twisting-twilled}
				Let $(\A_1\bowtie\A_2,\Pi)$ be a  twilled   Lie conformal algebra and $H:\A_2\longrightarrow\A_1$ a $\C[\partial]$-module homomorphism. Then $(\A_1\bowtie\A_2,\Pi^{H})$ is a twilled Lie conformal algebra if and only if $H$ is a  Maurer-Cartan element of  the differential graded Lie algebra $(C^*(\A_2,\A_1),d_{\hat{\mu}_2},[-,-]_{\hat{\mu}_1})$  given in Corollary \ref{pro:twilled GLA}, i.e.
				\begin{equation*}\label{eq:Twist-Twilled-MC}
				d_{\hat{\mu}_2}\hat{H}+\half[\hat{H},\hat{H}]_{\hat{\mu}_1}=0.
				\end{equation*}
				This is equivalent to
				\begin{equation*}
				\begin{split}\label{eq:MC-expression}
				&[H(u)_\lambda H(v)]_1+\rho_{2}(u)_\lambda H(v)-\rho_{2}(v)_{-\lambda-\partial}H(u)\\
				&\quad\quad\quad\quad\quad\quad\quad\quad\quad=H\big(\rho_{1}(H(u))_\lambda v-\rho_{1}(H(v))_{-\lambda-\partial}u\big)+H([u_\lambda v]_2),~\forall u,v\in\A_2.
				\end{split}
				\end{equation*}
			\end{prop}
			
			\begin{proof}A direct calculation gives
				\begin{align*}
				d_{\hat{\mu}_2}(\hat{H})(u,v)=&\rho_{2}(u)_\lambda H(v)-\rho_{2}(v)_{-\lambda-\partial}H(u)-H([u_\lambda v]_2),\\
				{[\hat{H},\hat{H}]_{\hat{\mu}_1}}(u,v)=&2H\big(\rho_{1}(H(u))_\lambda v-\rho_{1}(H(v))_{-\lambda-\partial}u\big)-2  [H(u)_\lambda H(v)]_1.
				\end{align*}
				Thus
				\begin{align*}
				d_{\hat{\mu}_2}\hat{H}(u,v)+\half[\hat{H},\hat{H}]_{\hat{\mu}_1}(u,v)=&\rho_{2}(u)_\lambda H(v)-\rho_{2}(v)_{-\lambda-\partial}H(u)-H([u_\lambda v]_2)\\
				&+H\big(\rho_{1}(H(u))_\lambda v-\rho_{1}(H(v))_{-\lambda-\partial}u\big)- [H(u)_\lambda H(v)]_1=0.
				\end{align*}
			\end{proof}
			
			\begin{coro}\label{cor:Lie and O-operator}
				Let $(\A_1\bowtie\A_2,\Pi)$ be a  twilled   Lie conformal algebra and $H$ a Maurer-Cartan element of the associated differential graded Lie algebra. Then we have a Lie conformal algebra structure on $\A_2$ given by
				\begin{eqnarray}\label{eq:mul1}
				[u_\lambda v]_H:=\rho_{1}(H(u))_\lambda v-\rho_{1}(H(v))_{-\lambda-\partial}u+[u_\lambda v]_2,~~\mbox{for}~u,v\in\A_2.
				\end{eqnarray}
			\end{coro}
			\begin{proof}
				By Lemma \ref{lem:twillL}, $\hat{\mu}_2^{H}$ is a Lie conformal algebra structure on $\A$. Furthermore, $\hat{\mu}_2^{H}$ restricted to $\A_2$ is a Lie conformal algebra structure and the $\lambda$-bracket on $\A_2$ is exactly \eqref{eq:mul1}.
			\end{proof}
			
			In the case of $\hat{\mu}_2=0$, $\A_1\bowtie\A_2$ is exactly a semi-direct product Lie conformal algebra $\A_1\ltimes\A_2$. To study this special case, we need the notion of a relative Rota-Baxter operator on a module over a Lie conformal algebra.
			
			\begin{defi}{\rm(\cite{HB2})
					Let $(M;\rho)$ be a  module over a Lie conformal algebra $\A$. A $\C[\partial]$-module homomorphism $T:M\rightarrow \A$ is called a  {\bf relative Rota-Baxter operator} (or an {\bf $\mathcal{O}$-operator}) on $M$ over $\A$ if it satisfies
					\begin{eqnarray}\label{eq:relative RB conditon}
					[T(m)_\lambda T(n)]=T\big(\rho(T(m))_\lambda n-\rho(T(n))_{-\la-\partial}m\big), ~~ \mbox{for}~~m,n\in M.
					\end{eqnarray}}
			\end{defi}
			
			Let us  denote the semi-direct product Lie conformal algebra structure on $\A\ltimes_\rho M$ by $\hat{\mu}$.
			\begin{prop}\label{pro:gLa-RB}
				A $\C[\partial]$-module homomorphism $T:M\rightarrow \A$ is  a relative Rota-Baxter operator if and only if $\hat{T}$ is a solution of the Maurer-Cartan equation in the graded Lie algebra $(C^*(M,\A),[-,-]_{\hat{\mu}})$ given in Corollary \ref{pro:twilled GLA}.
			\end{prop}
			\begin{proof} It follows from  $\half[\hat{T},\hat{T}]_{\hat{\mu}}(m,n)= [T(m)_\lambda T(n)]-T\big(\rho(T(m))_\lambda n-\rho(T(n))_{-\la-\partial}m\big).$
			\end{proof}
			
			\begin{coro}{\rm(\cite{HB2})}\label{cor:dualLCA}
				Let $T$ be a relative Rota-Baxter operator on the module $(M;\rho)$ over a Lie conformal algebra $\A$. Then $(M,[\cdot_\lambda\cdot]^T)$ is a Lie conformal algebra, where the $\lambda$-bracket $[\cdot_\lambda\cdot]^T$ is given by
				\begin{equation}\label{eq:Lie operation1}
				[m_\lambda n]^T= \rho(T(m))_{\lambda}n-\rho(T(n))_{-\lambda-\partial}m,~~\mbox{for}~m,n\in M,
				\end{equation}
				and $T$ is a Lie conformal algebra homomorphism from $(M,[\cdot_\lambda\cdot]^T)$ to $\A$.
			\end{coro}
			\begin{coro}\label{cor:RB-twilled Lie}
				Let $T$ be a relative Rota-Baxter operator on the module $(M;\rho)$ over a Lie conformal algebra $\A$. Then $(\A\oplus M,\Pi^T=\hat{\mu}_1^{T}+\hat{\mu}_2^{T})$ is a twilled Lie conformal algebra, where $\hat{\mu}_1^{T}=\hat{\mu},~	\hat{\mu}_2^{T}=[\hat{\mu},\hat{T}]_{\NR}$.
			\end{coro}
			
			Given a relative Rota-Baxter operator $T$ on the module $(M;\rho)$ over a Lie conformal algebra $\A$, we have a twilled Lie conformal algebra $\A\bowtie M_T$ by twisting $\A\ltimes_\rho M$ by $T$, where $M_T=(M,[\cdot_\lambda\cdot]^T)$ is the Lie conformal algebra defined by \eqref{eq:Lie operation1}. Define $\rho^T:M\rightarrow {\rm Cend}(\A)$ by
			\begin{eqnarray*}
				\label{eq:T-Rep}\rho^T(m)_{\lambda}a:=(\hat{\mu}_2^{T})_\lambda ((0,m),(a,0))=[T(m)_\lambda a]+T(\rho(a)_{-\lambda-\partial}m),~~\mbox{for}~a\in\A,~m\in M.
			\end{eqnarray*}
			By $[\hat{\mu}_2^{T},\hat{\mu}_2^{T}]_{\NR}=0$ and Proposition \ref{pro:MC}, $(\A; \rho^T)$ is a module of $M_{T}$. Thus, by \eqref{eq:twilled bracket}, we have the following:
			\begin{prop}
				The Lie conformal algebra structure on $\A\bowtie M_T$ is explicitly given by ($a,b\in\A,~m,n\in\ M$)
				\begin{equation}\label{eq:twilled bracket2}
				[(a,m)_\lambda (b,n)]_{\bowtie }=([a_\lambda b]+\rho^T(m)_{\lambda}b-\rho^T(n)_{-\lambda-\partial}a,[m_\lambda n]^T+\rho(a)_{\lambda}n-\rho(b)_{-\lambda-\partial} m).
				\end{equation}
			\end{prop}
			
			The notion of a conformal classical Yang-Baxter equation was introduced  in \cite{Li} in the study of Lie conformal bialgebras. Let $\A$ be a Lie conformal algebra and $r=\sum_{i} a_i\otimes b_i\in \A\otimes \A$. Set $\partial^{\otimes^{3}}=\partial\otimes1\otimes1+1\otimes\partial\otimes1+1\otimes1\otimes\partial.$  The equation
			\begin{equation}
			\begin{split}
			\llbracket r,r\rrbracket=&\sum_{i,j}\big([a_{i_{\mu}}a_{j}]\otimes b_{i}\otimes b_{j}| _{\mu=1\otimes \partial \otimes 1}-a_{i}\otimes[a_{j_{\mu}}b_{i}]\otimes b_{j}|_{\mu=1\otimes 1\otimes\partial}-a_{i}\otimes a_{j}\otimes[b_{j_{\mu}}b_{i}]|_{\mu=1\otimes\partial\otimes1}\big)\\
			\equiv&0~(\mbox{mod}~\partial^{\otimes^{3}})
			\end{split}
			\end{equation}
			is called the {\bf conformal classical Yang-Baxter equation} in $\A$.
			
			For any $r=\sum_{i} a_i\otimes b_i\in \A\otimes \A$, we associate a conformal linear map $r^{\sharp}\in {\rm Chom}(\A^{*c},\A)$ as follows:
			\begin{equation*}
			r^{\sharp}_{\lambda}(\alpha)=\sum_i\alpha_{-\lambda-\partial}(a_i)b_i,\quad \mbox{for}~\alpha\in \A^{*c}.
			\end{equation*}
			Set $r^{21}=\sum_i b_i\otimes a_i$. We say $r$ is {\bf skew-symmetric} if $r=-r^{21}$.
			
			Relative Rota-Baxter operators are closely related to the conformal classical Yang-Baxter equation, as the following lemma shows.
			
			\begin{lemm}{\rm(\cite{HB2})}\label{lem:equivalence r-matrix}
				Let $\A$ be a finite Lie conformal algebra. Then $r$ is a skew-symmetric solution of the conformal classical Yang-Baxter equation if and only if $r^{\sharp}_{0}=r^{\sharp}_{\lambda}|_{\lambda=0}$ is a relative Rota-Baxter operator on $(\A^{*c};\ad^*)$ over $\A$.
			\end{lemm}
			
			This, together with Proposition \ref{pro:gLa-RB}, gives the following result.
			\begin{prop}\label{44}
				Let $\A$ be a finite Lie conformal algebra. Then $r$ is a skew-symmetric solution of the  conformal classical Yang-Baxter equation if and only if $r^{\sharp}_{0}=r^{\sharp}_{\lambda}|_{\lambda=0}$ is a  Maurer-Cartan element of the graded Lie algebra $(C^*(\A^{*c},\A),[\cdot,\cdot]_{\hat{\mu}})$ given in Corollary \ref{pro:twilled GLA}, where ${\mu}$ is the Lie conformal algebra structure on $\A\ltimes_{\ad^*} \A^{*c}$.
			\end{prop}
			
			By Corollaries \ref{cor:dualLCA} and \ref{cor:RB-twilled Lie} and Proposition \ref{eq:twilled bracket2}, we have
			\begin{prop}
				Let $\A$ be a finite Lie conformal algebra. if $r$ is a skew-symmetric solution of the  conformal classical Yang-Baxter equation, then $ \A^{*c}_{r^\sharp}:=(\A^{*c},[\cdot_\lambda\cdot]^{r^\sharp})$ is a Lie conformal algebra, where the $\lambda$-bracket $[\cdot_\lambda\cdot]^{r^\sharp}$ is given by
				\begin{equation*}
				[\alpha_\lambda \beta]^{r^\sharp}= \ad^*(r_0^\sharp(\alpha))_\lambda\beta-\ad^*(r_0^\sharp(\beta))_{-\lambda-\partial}\alpha,\quad\forall~\alpha,\beta\in \A^{*c}.
				\end{equation*}
				Furthermore, the map $\fad^{*}:\A^{*c}\rightarrow {\rm gc}(\A)$ defined by
				$$\fad^{*}(\alpha)_\lambda a=[r_0^\sharp(\alpha)_\lambda a]+r_0^\sharp(\ad^*(a)_{-\lambda-\partial}\alpha),\quad\forall~\alpha\in \A^{*c},~a\in \A$$
				makes $\A$ into a module of $\A^{*c}_{r^\sharp}$, and there is a Lie conformal algebra structure on $\A\bowtie \A^{*c}_{r^\sharp}$ given by ($a,b\in\A,\alpha,\beta\in\A^{*c}$)
				\begin{equation*}
				[(a,\alpha)_\lambda (b,\beta)]_{\bowtie }=([a_\lambda b]+\fad^{*}(\alpha)_{\lambda}b-\fad^{*}(\beta)_{-\lambda-\partial}a,[\alpha_\lambda \beta]^{r^\sharp}+\ad^*(a)_\lambda \beta-\ad^*(b)_{-\lambda-\partial} \alpha).
				\end{equation*}
			\end{prop}
			\begin{remark}
				Given a skew-symmetric solution $r$ of the conformal classical Yang-Baxter equation on $\A$, the author in {\rm \cite{Li}} showed that there exists a Lie conformal algebra structure on $\A\oplus\A^{*c}$ corresponding to the coboundary Lie conformal bialgebra induced by $r$. Using our twisting theory, we also obtain this Lie conformal algebra structure on $\A\oplus\A^{*c}$ and give the concrete expression of this Lie conformal algebra structure directly.
			\end{remark}
			
			\subsection{The case of quasi-twilled Lie conformal algebras}
			Let $(\A,\A_1,\A_2)$ be a quasi-twilled Lie conformal algebra with the structure $\Pi=\hat{\mu}_1+\hat{\mu}_2+\hat{\phi}_1$. The twisted Lie conformal algebra structure $\Pi^H$ by $H:\A_2\rightarrow \A_1$ has the forms:
			\begin{eqnarray*}
				\hat{\phi}_1^{H}&=&\hat{\phi}_1,\\
				\hat{\mu}_1^{H}&=&\hat{\mu}_1+[\hat{\phi}_1,\hat{H}]_{\NR},\\
				\hat{\mu}_2^{H}&=&\hat{\mu}_2+d_{\hat{\mu}_1}\hat{H}+\half[[\hat{\phi}_1,\hat{H}]_{\NR},\hat{H}]_{\NR},\\
				\hat{\phi}_2^{H}&=&d_{\hat{\mu}_2}\hat{H}+\half[\hat{H},\hat{H}]_{\hat{\mu}_1}+\frac{1}{6}[[[\hat{\phi}_1,\hat{H}]_{\NR},\hat{H}]_{\NR},\hat{H}]_{\NR}.
			\end{eqnarray*}
			Recall that $C^*(\A_2,\A_1)$ has an $L_\infty$-algebra structure $(d_{\hat{\mu}_2},[-,-]_{{\hat{\mu}_1}},[-,-,-]_{\hat{\phi}_1})$ (see Theorem \ref{quasi-as-shLie}). Moreover, we have
			\begin{prop}\label{pro:quasi-twilled to quasi-twilled}
				The result of twisting $(\A,\A_1,\A_2,\Pi^{H})$ is also a quasi-twilled Lie conformal algebra  if and only if $H$ is a Maurer-Cartan element of the above $L_\infty$-algebra, i.e.
				\begin{equation*}\label{eq:Twist-quasi-Twilled-MC}
				d_{\hat{\mu}_2}(\hat{H})+\half[\hat{H},\hat{H}]_{{\hat{\mu}_1}}+\frac{1}{6}[\hat{H},\hat{H},\hat{H}]_{\hat{\phi}_1}=0.
				\end{equation*}
				This is equivalent to
				\begin{align*}\label{eq:MC-expression2}
				[H(u)_\lambda H(v)]_1&+\rho_{2}(u)_\lambda H(v)-\rho_{2}(v)_{-\lambda-\partial}H(u)\\
				\nonumber\quad\quad	&=H\big(\rho_{1}(H(u))_\lambda v-\rho_{1}(H(v))_{-\lambda-\partial}u\big)+H([u_\lambda v]_2)+H(\phi_{1\lambda}(H(u),H(v))), ~~\mbox{for}~~u,v\in \A_2.
				\end{align*}
			\end{prop}
			\begin{proof}
				It follows by a direct calculation.
			\end{proof}
			
			\begin{coro}\label{cor:new LCA}
				Let $(\A,\A_1,\A_2)$ be a  quasi-twilled   Lie conformal algebra and $H$ a solution of the Maurer-Cartan equation in the associated $L_\infty$-algebra. Then
				\begin{eqnarray}\label{eq:mul3}
				[u_\lambda v]^H:=\rho_{1}(H(u))_\lambda v-\rho_{1}(H(v))_{-\lambda-\partial}u+[u_\lambda v]_2+\phi_{1\lambda}(H(u),H(v)),~~\mbox{for}~u,v\in\A_2
				\end{eqnarray}
				defines a  Lie conformal algebra structure on $\A_2$.
			\end{coro}
			\begin{proof}
				By Lemma \ref{lem:twillL}, we deduce that $\hat{\mu}_2^{H}$ restrited to $\A_2$ is a  Lie conformal structure on $\A_2$. Furthermore, $\hat{\mu}_2^{H}$ restricted to $\A_2$ is exactly \eqref{eq:mul3}.
			\end{proof}
			In the case of $\hat{\mu}_2=0$, the quasi-twilled Lie conformal algebra $(\A,\A_1,\A_2)$ is a $\phi$-twisted semi-direct product Lie conformal algebra $\A_1\ltimes_{\phi}\A_2$. In the following, we propose the notion of a twisted relative Rota-Baxter operator on a Lie conformal algebra, which can be used to twist a quasi-twilled Lie conformal algebra.
			
			\begin{defi}{\rm
					Let $(\A,[\cdot_\lambda\cdot])$ be a Lie conformal algebra, $(M;\rho)$ an $\A$-module and $\phi$ a $2$-cocycle  in $C^2(\A, M)$.  A $\C[\partial]$-module homomorphism $T:M\rightarrow \A$ is called a  {\bf $\phi$-twisted relative Rota-Baxter operator } 
					if it satisfies
					\begin{eqnarray}\label{eq:twisted relative RB conditon}
					[T(m)_\lambda T(n)]=T\big(\rho(T(m))_\lambda n-\rho(T(n))_{-\la-\partial}m+\phi_\lambda (T(m),T(n))\big),~~\mbox{for}~ m,n\in M.
					\end{eqnarray}}
			\end{defi}

			Let $(M;\rho)$ be a module over a Lie conformal algebra $\A$ and $\phi$ a $2$-cocycle associated to the module $(M;\rho)$. Let $\hat{\mu}+\hat{\phi}$ denote the structure of the quasi-twilled Lie conformal algebra $\A\ltimes_{\phi} M$.
			\begin{prop}\label{pro:TRB-MC}
				A $\C[\partial]$-module homomorphism $T:M\rightarrow \A$ is  a $\phi$-twisted relative Rota-Baxter operator 
				if and only if the lift $\hat{T}$ is a  Maurer-Cartan element of the $L_\infty$-algebra $(C^*(M,\A),d_{\hat{\mu}_2}=0,[-,-]_{{\hat{\mu}_1}},[-,-,-]_{\hat{\phi}_1})$  given in Theorem \ref{quasi-as-shLie}, namely, $$\half[\hat{T},\hat{T}]_{\hat{\mu_1}}+\frac{1}{6}[\hat{T},\hat{T},\hat{T}]_{\hat{\phi}_1}=0,$$
				where $\mu_1=\mu$, $\mu_2=0$ and $\phi_1=\phi$.
			\end{prop}
			\begin{proof}
				It follows from Proposition \ref{pro:quasi-twilled to quasi-twilled}.
			\end{proof}
			
			The author in \cite{Getzler} constructed a new $L_{\infty}$-algebra from an old $L_{\infty}$-algebra along with a  Maurer-Cartan element. As a direct application, we have
			\begin{prop}\label{pro:twisted L_infty}
				Let $T$ be a $\phi$-twisted relative Rota-Baxter operator on a module $(M;\rho)$ over a Lie conformal algebra $\A$. Then  $(C^*(M,\A),l_1^T,l_2^T,l_3^T)$ is an $L_{\infty}$-algebra with trivial higher brackets, where
				$$l_1^T(f)=[\hat{T},\hat{f}]_{{\hat{\mu}_1}}+\half[\hat{T},\hat{T},\hat{f}]_{\hat{\phi}_1},\quad l_2^T(f,g)=[\hat{f},\hat{g}]_{{\hat{\mu}_1}}+[\hat{T},\hat{f},\hat{g}]_{\hat{\phi}_1},\quad l_3^T(f,g,h)=[\hat{f},\hat{g},\hat{h}]_{\hat{\phi}_1},$$
				for all $f\in C^m(M,\A),~g\in C^n(M,\A),~h\in C^k(M,\A).$
				We call this a {\bf $T$-twisted $L_{\infty}$-algebra}.
			\end{prop}
			
			\begin{prop}
				Let $T:M\rightarrow \A$ be a $\phi$-twisted relative Rota-Baxter operator.
				Then for any $\C[\partial]$-module homomorphism $T':M\rightarrow \A$, $T+T'$ is also a $\phi$-twisted relative Rota-Baxter operator  if and only $T'$ is a  Maurer-Cartan element of the $T$-twisted $L_{\infty}$-algebra from Proposition \ref{pro:twisted L_infty}.
			\end{prop}
			\begin{proof}
				It follows by a direct calculation.
			\end{proof}	
			
			By Proposition \ref{pro:TRB-MC} and Corollary \ref{cor:new LCA}, we have
			\begin{coro}\label{cor:twist-LCA2}
				Let $T$ be a $\phi$-twisted relative Rota-Baxter operator. Then 
				\begin{itemize}
					\item[{\rm(i)}]
					$(M,[\cdot_\lambda\cdot]^{T,\phi})$ is a Lie  conformal algebra, where the $\lambda$-bracket $[\cdot_\lambda\cdot]^{T,\phi}$ is given by
					\begin{equation}\label{eq:LCA2}
					[m_\lambda n]^{T,\phi}=\rho(T(m))_\lambda n-\rho(T(n))_{-\lambda-\partial}m+\phi_\lambda(T(m),T(n)), ~~\mbox{for}~m,n\in M.
					\end{equation}
					Furthermore, $T$ is a Lie conformal algebra homomorphism from $(M,[\cdot_\lambda\cdot]^{T,\phi})$ to $\A$.
					\item[{\rm (ii)}] $(\A\oplus M,\Pi^T=\hat{\mu}_1^{T}+\hat{\mu}_2^{T}+\hat{\phi}_1^{T})$ is a quasi-twilled Lie conformal algebra, where $$\hat{\mu}_1^{T}=\hat{\mu},~~	\hat{\mu}_2^{T}=[\hat{\mu},\hat{T}]_{\NR}+\half[[\hat{\phi},\hat{T}]_{\NR},\hat{T}]_{\NR},~~ \hat{\phi}_1^{T}=\hat{\phi}.$$
				\end{itemize}
			\end{coro}
			
			
			Given a $\phi$-twisted  relative Rota-Baxter operator $T$ on a module $(M;\rho)$ over a Lie conformal algebra $\A$, we have a quasi-twilled Lie conformal algebra $\A\bowtie_\phi M_{T,\phi}$ by twisting $\A\ltimes_\phi M$ by $T$, where $M_{T,\phi}=(M,[\cdot_\lambda\cdot]^{T,\phi})$ is
			given by \eqref{eq:LCA2}. Define $\rho^T:M\rightarrow {\rm Cend}(\A)$ by
			\begin{eqnarray}
			\label{eq:T-Rep2}\rho^T(m)_{\lambda}a:=(\hat{\mu}_2^{T})_\lambda ((0,m),(a,0))
			=[T(m)_\lambda a]+T(\rho(a)_{-\lambda-\partial}m)-T(\phi_\lambda(T(m),a)),
			\end{eqnarray}
			for $a\in\A,~m\in M.$  By $[\hat{\mu}_2^{T},\hat{\mu}_2^{T}]_{\NR}=0$ and Proposition \ref{pro:MC}, $(\A; \rho^T)$ is a module of 
			$M_{T,\phi}$. Furthermore, by \eqref{eq:quasi-twilled bracket}, we have
			
			\begin{prop}
				The Lie conformal algebra structure on $\A\bowtie_\phi M_{T,\phi}$ is explicitly given by
				\begin{align*}\label{eq:twilled bracket3}
				{[(a,m)_\lambda (b,n)]_{\bowtie_\phi  }}=&([a_\lambda b]+\rho^T(m)_{\lambda}b-\rho^T(n)_{-\lambda-\partial}a,[m_\lambda n]^{T,\phi}+\rho(a)_{\lambda}n-\rho(b)_{-\lambda-\partial} m\\
				\nonumber	&+\phi_\lambda(H(m),b)+\phi_\lambda(a,H(n))+\phi_\lambda(a,b)),
				\end{align*}
				for $a,b\in\A$ and $m,n\in\ M.$
			\end{prop}
			
			Assume that $T:M\rightarrow \A$ is a $\C[\partial]$-module homomorphism. We denote the graph of $T$ by ${\rm Gr}(T)$,
			\begin{align*}
			{\rm Gr}(T)=\{(T(m),m)|m\in M\}.
			\end{align*}
			\begin{prop}\label{p11}
				Let $(M;\rho)$ be a module over a Lie conformal algebra $\A$ and $\phi$ a $2$-cocycle in $C^2(\A, M)$. Then $T:M\rightarrow \A$ is a $\phi$-twisted relative Rota-Baxter operator if and only if ${\rm Gr}(T)$ is a subalgebra of $\A\ltimes_{\phi}M$.
			\end{prop}
			\begin{proof} For any $(T(m),m)$, $(T(n),n)\in {\rm Gr}(T)$, we have
				\begin{align*}
				[(T(m),m)_\lambda (T(n),n)]=\big([T(m)_\lambda T(n)], \rho(T(m))_\lambda n-\rho(T(n))_{-\lambda-\partial} m+\phi_\lambda(T(m),T(n))\big).
				\end{align*}
				Hence $T$ is a $\phi$-twisted relative Rota-Baxter operator if and only if $\big([T(m)_\lambda T(n)], \rho(T(m))_\lambda n-\rho(T(n))_{-\lambda-\partial} m++\phi_\lambda(T(m),T(n))\big)$ is in ${\rm Gr}(T)[\lambda]$.
			\end{proof}
			
			\begin{ex} {\rm Let $\omega:\A \rightarrow M$ be an invertible $1$-cochain in $C^1(\A, M)$. Set $\phi=-\mathbf{d}\omega$. Then the inverse $\omega^{-1}:M\rightarrow\A$ is a $\phi$-twisted relative Rota-Baxter operator. In fact, putting $T=\omega^{-1}$, the condition \eqref{eq:twisted relative RB conditon} is equivalent to
					\begin{eqnarray*}
						\omega([T(m)_\lambda T(n)])=\rho(T(m))_\lambda n-\rho(T(n))_{-\lambda-\p}m +\phi_\lambda(T(m),T(n)), ~~ \forall ~~ m,n\in M.
					\end{eqnarray*}
					This is the same as
					\begin{align*}
					\phi_\lambda(T(m),T(n)) =-\rho(T(m))_\lambda n+\rho(T(n))_{-\lambda-\p}m+\omega([T(m)_\lambda T(n)])=-(\mathbf{d}\omega)_\lambda(T(m),T(n)).
					\end{align*}
				}\end{ex}
				\begin{ex} {\rm
						Let $\A$ be a Lie conformal algebra and  $\phi\in C^2(\A,\A)$ defined by $\phi_\lambda(a,b)=-[a_\lambda b]$, for $a,b\in\A$.
						Since $\phi$ is a $2$-cocycle, ${\rm Id}: \A\rightarrow \A$ is a $\phi$-twisted relative Rota-Baxter operator.}\end{ex}

				\begin{defi}{\rm \label{def3}
						Let $\A$ be a Lie conformal algebra. A $\C[\partial]$-module homomorphism $N:\A\rightarrow \A$ is called a {\bf Nijenhuis operator} on $\A$ if the following condition is satisfied for all $a,b\in \A$:
						\begin{eqnarray}\label{N}
						[N(a)_\lambda N(b)]=N\big([N(a)_\lambda b]+[a_\lambda N(b)]-N[a_\lambda b]\big).
						\end{eqnarray}
					}\end{defi}
					
					\begin{lemm}{\rm(\cite{LZ})}\label{lem:Nijenhuis}
						Let $N:\A\rightarrow \A$ be a Nijenhuis operator on a Lie  conformal algebra $(\A,[\cdot_\lambda \cdot])$. Then $\A^N:=(\A,[\cdot_\lambda \cdot]_N)$ is also a Lie conformal algebra, where $[\cdot_\lambda \cdot]_N$ is defined by
						\begin{equation}
						[a _\lambda b]_N=[N(a)_\lambda b]+[a_\lambda N(b)]-N[a_\lambda b], ~ \mbox{for}  ~ a,b\in \A,
						\end{equation}
						and $N$ is a Lie conformal algebra homomorphism from $\A^N$  to $\A$.
					\end{lemm}		
					
					\begin{ex}\label{Nijenhuis-Rota}
						{\rm Let $N:\A\rightarrow \A$ be a Nijenhuis operator on a Lie conformal algebra $\A$. By Lemma \ref{lem:Nijenhuis}, $N$ induces a new Lie conformal algebra structure $\A^N$. Then $\A$ becomes an $\A^N$-module by
							\begin{align*}
							\rho(a)_\la x=[N(a)_\la x], ~~ \mbox{for}~ a\in\A^N,~  x\in\A.
							\end{align*}
							With this module, the map $\phi_\la (a,b):=-N[a_\la b]$ is a $2$-cocycle in $C^2(\A^N,\A)$. Then the identity map ${\rm Id}: \A\rightarrow \A^N$ is a $\phi$-twisted relative Rota-Baxter operator.
						}\end{ex}

							\begin{defi} {Let $\A$ be a Lie conformal algebra. A $\C[\partial]$-module homomorphism $R:\A\rightarrow \A$ is called a {\bf Reynolds operator} on $\A$ if it satisfies
									\begin{eqnarray}\label{R}
									[R(a)_\lambda R(b)]=R\big([R(a)_\lambda b]+[a_\lambda R(b)]-[R(a)_\lambda R(b)]\big), ~~\mbox{for}~a,b\in \A.
									\end{eqnarray}
								}\end{defi}
								
								Notice that the last term  $-[R(a)_\lambda R(b)]$ in \eqref{R} is the $\lambda$-bracket on $\A$, which is a $2$-cocycle. Therefore, each Reynolds operator $R$ can be seen as a $\phi$-twisted relative Rota-Baxter operator with $\phi_\la(a,b):=-[a_\la b]$ for all $a,b\in \A$. It follows from \eqref{eq:LCA2} that $R$ gives rise to a new Lie conformal algebra structure, denoted by $\A^R$, on $\A$ given by
								\begin{align*}
								[a_\la b]^R=[R(a)_\lambda b]+[a_\lambda R(b)]-[R(a)_\lambda R(b)], ~~\mbox{for}~a,b\in \A.
								\end{align*}
								By \eqref{R}, $R$ is a Lie conformal algebra homomorphism from $\A^R$ to $\A$. If $R$ is invertible, it follows from \eqref{R} that
								\begin{eqnarray*}\label{R-1}
								R^{-1}([a_\lambda b])=[R^{-1}(a)_\lambda b]+[a_\lambda R^{-1}(b)]-[a_\lambda b], ~~\mbox{for}~ a,b\in\A.
								\end{eqnarray*}
								So $(R^{-1}-{\rm Id})([a_\lambda b])=[(R^{-1}-{\rm Id})(a)_\lambda b]+[a_\lambda (R^{-1}-{\rm Id})(b)]$. Namely, $R^{-1}-{\rm Id}:\A\rightarrow \A$ is a {1-cocycle}. Conversely, if $d:\A\rightarrow\A$ is a {1-cocycle} such that ${\rm Id}+d$ is invertible, then $({\rm Id}+d)^{-1}$ is a Reynolds operator on $\A$. Even if ${\rm Id}+d$ is not invertible but $({\rm Id}+d)^{-1}=\sum_{n=0}^\infty (-1)^nd^n$ converges pointwise, $({\rm Id}+d)^{-1}$ is again a Reynolds operator on $\A$. A more precise statement is given below by a verbatim repetition of the proof of \cite[Proposition 2.8]{ZGG} in terms of $\la$-bracket.
								
								\begin{prop}
									Let $\A$ be a Lie conformal algebra with a $1$-cocycle $d\in C^1(\A,\A)$. If the series $\sum_{n=0}^\infty (-1)^nd^n (x)$ is convergent for all $x\in\A$, then $\sum_{n=0}^\infty (-1)^nd^n$ is a Reynolds operator on $\A$. This is the case when $d$ is locally nilpotent.
								\end{prop}

								A Reynolds operator on a Lie conformal algebra is said to be  {\bf nontrivial} if it is neither invertible nor equal to $0$. Otherwise, it is called {\bf trivial}.
								\begin{ex}  {\rm The Virasoro Lie conformal algebra ${\rm Vir}=\C[\partial] L$ is a free $\C[\partial]$-module generated by the symbol $L$, satisfying
										$[L_\la L]=(\partial+2\la) L.$ Assume that $R$ is Reynolds operator on ${\rm Vir}$. Write $R(L)=f(\partial)L$ for some $f(\partial)\in \C[\partial]$. Substituting this into the defining relation of Reynolds operator gives
										\begin{align*}					f(-\la)f(\partial+\la)=(f(-\la)+f(\partial+\la)-f(-\la)f(\partial+\la))f(\partial).
										\end{align*}
										Equating terms of highest degree in $\partial$ in both sides of the above equation, we obtain that $f$ is constant. So $R$ is trivial.
									}\end{ex}
\section{NS-Lie conformal algebras}
In this section, we introduce the notion of an NS-Lie conformal algebra, which is a conformal analogue of NS-Lie algebras given in  \cite{Das2}.  We show that NS-Lie conformal algebras connect closely with Lie conformal algebras, conformal NS-algebras, twisted relative Rota-Baxter operators and Nijenhuis operators.
Various examples of NS-Lie conformal algebras are given.

\begin{defi}{\rm \label{defi5} { Let $\mathcal{A}$ be a $\C[\partial]$-module equipped with two binary $\lambda$-multiplications	$\circ_\lambda$ and $\vee_\lambda$. Then $\mathcal{A}$ is called an {\bf  NS-Lie conformal algebra}, if $\circ_\lambda$ and $\vee_\lambda$ are conformal sesquilinear maps, and $\vee_\lambda$ is skew-symmetric, i.e.,
	\begin{align}
		a \vee_\lambda b=-b \vee_{-\lambda-\p} a, \ \forall \ \ a,b\in \mathcal{A},
\end{align}
			and the following axioms hold for all $a,b,c\in\mathcal{A}$:
	\begin{align}
&(a\circ_\la b)\circ_{\la+\mu} c-a\circ_\la(b\circ_\mu c)-(b\circ_\mu a)\circ_{\la+\mu}c+b\circ_\mu(a\circ_\la c)+(a\vee_\la b)\circ_{\la+\mu} c=0,\label{NS1}\\
	& a\vee_\la [b_\mu c]-[a_\la b]\vee_{\la+\mu} c-b\vee_\mu [a_\la c]+a\circ _\la(b\vee_\mu c)-b\circ_\mu (a\vee_\la c)+c\circ_{-\la-\mu-\p}(a\vee_\la b)=0,\label{NS2}
	\end{align}
	where the $\lambda$-bracket $[\cdot_\lambda\cdot]$ is defined by
	\begin{align}\label{NS5-5}
	[a_\lambda b]=a\circ_ \lambda b-b\circ_{-\lambda-\p} a+a\vee_\lambda b, \ \forall \ a,b\in \mathcal{A}.
	\end{align}}}
	\end{defi}
									
	\begin{remark}{  If the binary $\lambda$-multiplication $\circ_\lambda$ in Definition \ref{defi5} is trivial, then $(\mathcal{A},\vee_\la)$ is a Lie conformal algebra.
	If $\vee_\la$ is trivial, then $(\mathcal{A},\circ_\la)$ becomes a left-symmetric conformal algebra. Hence NS-Lie conformal algebras are a generalization of both Lie conformal algebras and left-symmetric conformal algebras. See {\rm \cite{HL}} for details about left-symmetric conformal algebras.
	}\end{remark}

By conformal sesquilinearity, \eqref{NS1} is equivalent to the following identity:
\begin{equation}
\begin{split}
\label{NS1-1} 0=&(a\circ_\la b)\circ_{-\mu-\p} c-a\circ_\la(b\circ_{-\mu-\p}c)-(b\circ_{-\la-\p} a)\circ_{-\mu-\p}c+b\circ_{-\la-\mu-\p}(a\circ_\la c)\\&+(a\vee_\la b)\circ_{-\mu-\p} c.
\end{split}
\end{equation}
\begin{defi}
Let $(\A,\circ_\lambda,\vee_\la)$ and $(\B,\circ'_\lambda,\vee'_\la)$ be two $\NS$-Lie  conformal algebras. A {\bf homomorphism}  from $\A$ to $\B$ is a $\C[\partial]$-module homomorphism $\phi$ satisfying
$$\phi(a\circ_\lambda b)=\phi(a)\circ'_\lambda \phi(b),\quad \phi(a\vee_\lambda b)=\phi(a)\vee'_\lambda \phi(b),\quad\forall~a,b\in\A.$$
\end{defi}


\begin{theo}\label{p5}  Let $(\mathcal{A},\circ_\la, \vee_\la)$ be an $\NS$-Lie conformal algebra. Then $(\mathcal{A}, [\cdot_\lambda\cdot])$ is a Lie conformal algebra, where the $\lambda$-bracket $[\cdot_\lambda\cdot]$ is given by \eqref{NS5-5}, which is called the {\bf sub-adjacent Lie conformal algebra} of  $(\mathcal{A},\circ_\la, \vee_\la)$ and denoted by  $\mathcal{A}_{Lie}$. Furthermore, the $\la$-action of $\mathcal{A}_{Lie}$ on $\A$ defined by
\begin{equation}\label{eq:module}
\rho(a)_\la x= a \circ_\la x, \ \mbox{for}\ a\in \mathcal{A}_{Lie},\ x\in\mathcal{A}
\end{equation}
gives a module of $\mathcal{A}_{Lie}$ on $\A$.
\end{theo}

\begin{proof} It is easy to see that the $\lambda$-bracket $[\cdot_\lambda\cdot]$ is conformal sesquilinear and skew-symmetric. To check the Jacobi identity, we compute separately,
\begin{align*}
[a_\la[b_\mu c]]=&[a_\la(b\circ_\mu c-c\circ_{-\mu-\p}b+b\vee_\mu c)]\nonumber\\
=&a\circ_\la (b\circ_\mu c-c\circ_{-\mu-\p}b+b\vee_\mu c)-(b\circ_\mu c-c\circ_{-\mu-\p}b+b\vee_\mu c)\circ_{-\la-\p}a+a\vee_\la (b\ast_\mu c),\nonumber\\
[b_\mu[a_\la c]]=&b\circ_\mu (a\circ_\la c-c\circ_{-\la-\p}a+a\vee_\la c)-(a\circ_\la c-c\circ_{-\la-\p}a+a\vee_\la c)\circ_{-\mu-\p} b+b\vee_\mu(a\ast_\la c),\nonumber\\
[[a_\la b]_{\la+\mu}c]=&[(a\circ_\la b-b\circ_{-\la-\p}a+a\vee_\la b)_{\la+\mu}c]\nonumber\\
=&(a\circ_\la b-b\circ_{-\la-\p}a+a\vee_\la b)\circ _{\la+\mu}c-c\circ _{-\la-\mu-\p}(a\circ_\la b\nonumber\\&-b\circ_{-\la-\p}a+a\vee_\la b)+(a\ast_\la b)\vee_{\la+\mu}c.\nonumber
\end{align*}
Thus, we have
\begin{equation}
\begin{split}
\label{NS5-9}	&[[a_\la b]_{\la+\mu}c]+[b_\mu[a_\la c]]-[a_\la[b_\mu c]]\\=&\big((a\circ_\la b-b\circ_{-\la-\p}a+a\vee_\la b)\circ _{\la+\mu}c-a\circ_\la(b\circ_\mu c)+b\circ_\mu(a\circ_\la c)\big)\\
&+\big((b\circ_\mu c-c\circ_{-\mu-\p}b+b\vee_\mu c)\circ_{-\la-\p}a-b\circ_\mu(c\circ_{-\la-\p}a)+c\circ_{-\la-\mu-\p}(b\circ_{-\la-\p}a)\big)\\
&-\big( (a\circ_\la c-c\circ_{-\la-\p}a+a\vee_\la c)\circ_{-\mu-\p} b -a\circ_\la(c\circ_{-\mu-\p}b)-c\circ_{-\la-\mu-\p}(a\circ_\la b)\big)\\&
+\big((a\ast_\la b)\vee_{\la+\mu}c+b\vee_\mu(a\ast_\la c)-a\vee_\la (b\ast_\mu c)-c\circ_{-\la-\mu-\p}(a\vee_\la b)  \\
& +b\circ_\mu(a\vee_\la c)-a\circ_\la(b\vee_\mu c)\big).
\end{split}
\end{equation}
The first term in the RHS of \eqref{NS5-9} vanishes due to \eqref{NS1} and \eqref{NS3}. The second and third terms in the RHS of \eqref{NS5-9} vanish due to \eqref{NS4-4} and \eqref{NS1-1}. The last term in the RHS of \eqref{NS5-9} vanishes due to \eqref{NS2}. This proves that $(\mathcal{A}, [\cdot_\lambda\cdot])$ is a Lie conformal algebra.

By conformal sesquilinearity of the $\lambda$-operation $\circ_\la$, \eqref{eq:module} satisfies the second condition of \eqref{module}. By using \eqref{NS3} and \eqref{NS5-5}, \eqref{NS1} can be written as
$$[a_\lambda b]\circ _{\la+\mu}c-a\circ_\la(b\circ_\mu c)+b\circ_\mu(a\circ_\la c)=0,$$
which implies that \eqref{eq:module} satisfies the first condition of \eqref{module}. This ends the proof.
\end{proof}

We introduced conformal NS-algebras in \cite{YUAN} by analogy with Leroux's NS-algebras \cite{L}.
\begin{defi}{\rm Let $\mathcal{A}$ be a $\C[\partial]$-module equipped with three binary $\lambda$-multiplications
$\succ_\lambda, \prec_\lambda$ and $\curlyvee_\lambda$. Then $\mathcal{A}$ is called a {\bf conformal NS-algebra}, if $\succ_\lambda, \prec_\lambda$ and $\curlyvee_\lambda$ are conformal sesquilinear maps, and satisfy the following axioms for all $x,y,z\in\mathcal{A}$:
\begin{align}
x\succ_\lambda(y\succ_\mu z)&=(x\times_\lambda y)\succ_{\lambda+\mu}z,\label{NS11}\\
x\prec_\lambda (y\times_\mu z)&=(x\prec_\lambda y)\prec_{\lambda+\mu}z,\label{NS22}\\
x\succ_\lambda (y\prec_\mu z)&=(x\succ_\lambda y)\prec_{\lambda+\mu}z,\label{NS33}\\
x\succ_\lambda (y\curlyvee_\mu z)- (x\times_\lambda y)\curlyvee_{\lambda+\mu}z&=(x\curlyvee_\lambda y)\prec_{\lambda+\mu}z-x\curlyvee_\lambda (y\times_\mu z),\label{NS44}
\end{align}
where $\times_\lambda$ is defined as
\begin{align}\label{NS55}
x\times_\lambda y=x\succ_\lambda y+x\prec_\lambda y+x\curlyvee_\lambda y.
\end{align}
}\end{defi}

\begin{theo}\label{th1} Let $(\A,\succ_\lambda, \prec_\lambda, \curlyvee_\lambda)$ be a conformal NS-algebra. Then $(\A,\circ_\la,\vee_\la)$
forms an NS-Lie conformal algebra, where
\begin{align}\label{NS8}
x\circ_\la y=x\succ_\lambda y-y\prec_{-\lambda-\p}x,\ \ x\vee_\la y=x\curlyvee_{\lambda} y-y\curlyvee_{-\lambda-\p}x, ~~\mbox{for}~x,y\in\A.
\end{align}
\end{theo}
\begin{proof}
It follows by a direct calculation. We omit the details.
\end{proof}

The following theorem reveals a close connection between twisted relative Rota-Baxter operators and NS-Lie conformal algebras.

\begin{theo}\label{th7} Assume that $M$ is a module over a Lie conformal algebra $\A$, $\phi$ is a $2$-cocycle in $C^2(\A, M)$, and $T:M\rightarrow \A$ is a
$\phi$-twisted relative Rota-Baxter operator. Then $M$ becomes an NS-Lie conformal algebra under the following $\la$-multiplications:
\begin{eqnarray}\label{6-6}
u\circ_\la v=T(u)_\la v, ~~ u\vee_\la v=\phi_\la(T(u),T(v)), ~~\mbox{for}~u,v\in M.
\end{eqnarray}
\end{theo}
\begin{proof} It is easy to check that $\circ_\lambda$ and $\vee_\lambda$ are conformal sesquilinear, and $\vee_\lambda$ is skew-symmetric. For $u,v,w\in M$, we have
\begin{align*}
(u\circ_\la v)&\circ_{\la+\mu} w-u\circ_\la(v\circ_\mu w)-(v\circ_\mu u)\circ_{\la+\mu}w+v\circ_\mu(u\circ_\la w)\\
&\stackrel{\eqref{6-6}}{=} T(T(u)_\la v)_{\la+\mu}w-T(u)_\la(T(v)_\mu w)-T(T(v)_\mu u)_{\la+\mu}w+T(v)_\mu(T(u)_\la w)\\
&\stackrel{\eqref{module}}{=} T(T(u)_\la v)_{\la+\mu}w-T(T(v)_\mu u)_{\la+\mu}w-[T(u)_\la T(v)]_{\la+\mu}w\\
&\stackrel{\eqref{NS3}}{=} T(T(u)_\la v)-T(v)_{-\la-\p} u))_{\la+\mu}w-[T(u)_\la T(v)]_{\la+\mu}w\\
&\stackrel{\eqref{eq:LCA2}}{=} -T\phi_\lambda(T(u), T(v))_{\la+\mu}w\stackrel{\eqref{6-6}}{=}-(u\vee_\la v)\circ_{\la+\mu}w.
\end{align*}
Hence \eqref{NS1} is valid. To show \eqref{NS2}, we start with the cocycle condition of $\phi$:
\begin{align*}
0=&T(u)_\lambda \phi_\mu (T(v),T(w))-T(v)_\mu \phi_\lambda (T(u),T(w)) +T(w)_{-\la-\mu-\partial}\phi_{\lambda}(T(u),T(v))\\&+\phi_{\lambda}(T(u), [T(v)_\mu T(w)])-\phi_{\mu}(T(v), [T(u)_\lambda T(w)])-\phi_{\la+\mu}([T(u)_{\lambda}T(v)],T(w)).
\end{align*}
This, together with \eqref{6-6}, gives
\begin{equation}
\begin{split}\label{7-7}
0=&u\circ_\lambda (v\vee_\mu w)-v\circ_\mu (u\vee_\la w)+w\circ_{-\la-\mu-\partial}(u\vee_\la v)\\
&+\phi_{\lambda}(T(u), [T(v)_\mu T(w)])-\phi_{\mu}(T(v), [T(u)_\lambda T(w)])-\phi_{\la+\mu}([T(u)_{\lambda}T(v)],T(w)).
\end{split}
\end{equation}
On the other hand, we have
\begin{align*}
[T(u)_\lambda T(v)]&=T\big( T(u)_\la v-T(v)_{-\la-\p}u +\phi_\la(T(u),T(v)) \big)\\&=T(u\circ_\la v-v\circ_{-\la-\p} u+u\vee_\la v)=T(u\ast_\la v).
\end{align*}
It follows that
\begin{align*}
\phi_{\lambda}(T(u), [T(v)_\mu T(w)])&=\phi_{\lambda}(T(u), T(v\ast_\mu w))=u\vee_{\la}(v\ast_\mu w),\\
\phi_{\mu}(T(v), [T(u)_\lambda T(w)])&=\phi_{\mu}(T(v), T(u\ast_\la w))=v\vee_\mu(u\ast_\la w),\\				\phi_{\la+\mu}([T(u)_{\lambda}T(v)],T(w))&=\phi_{\la+\mu}(T(u\ast_{\lambda}v),T(w))=(u\ast_{\lambda}v)\vee_{\la+\mu}w.
\end{align*}
Plugging this back into \eqref{7-7}, we obtain \eqref{NS2}.
\end{proof}

It follows immediately from  Theorems \ref{p5} and \ref{th7} that if $T:M\rightarrow \A$ is a $\phi$-twisted relative Rota-Baxter operator, then $T$ equips $M$ with a structure of Lie conformal algebra, which is exactly the one defined by \eqref{eq:LCA2}.

Let $\A$ be a Lie conformal algebra with a module $(M;\rho)$, and $\A'$ a Lie conformal algebra with a module $(M';\rho')$. Suppose that $T: M \rightarrow\A$ is a $\phi$-twisted relative Rota-Baxter operator and $T^\prime:M'\rightarrow\A'$ is a $\phi^{\prime}$-twisted relative Rota-Baxter operator, where $\phi$ and $\phi^{\prime}$ are $2$-cocycles in $C^2(\A,M)$ and $C^2(\A',M')$, respectively.
\begin{defi}{ A {\bf morphism} of twisted relative Rota-Baxter operators from $T$ to $T^\prime$ consists of a pair $(\chi, \psi)$ of a Lie conformal algebra homomorphism $\chi:\A\rightarrow \A'$ and a $\C[\partial]$-module homomorphism $\psi:M\rightarrow M'$ satisfying
\begin{align}\label{morphism}
\chi\circ T=T^\prime\circ \psi,\ \
\rho'(\chi(a))_\la \psi(m)=\psi(\rho(a)_\la m),\ \
\psi\phi_\la(a,b)=\phi^\prime_\la(\chi(a),\chi(b)),
\end{align}
for all $a,b\in\A$ and $m\in M$. It is called an  {\bf isomorphism} if $\phi$ and $\psi$ are both linear isomorphisms.
}\end{defi}

\begin{prop}\label{p9} With notations above. If $(\chi, \psi)$ is a morphism from $T$ to $T^\prime$, then $\psi:M\rightarrow M'$ is an $\NS$-Lie conformal algebra homomorphism from $(M,\circ_\la,\vee_\lambda)$ to $(M',\circ'_\la,\vee'_\lambda)$, where $\circ_\la,~ \vee_\lambda, ~\circ'_\la$ and $\vee'_\lambda$ are given by
\begin{align*}
u\circ_\la v=T(u)_\la v,& \quad u\vee_\la v=\phi_\la(T(u),T(v)),~~\mbox{for}~u,v\in M;\\
u'\circ'_\la v'=T'(u')_\la v',& \quad u'\vee'_\la v'=\phi_\la(T'(u'),T'(v')), ~~\mbox{for}~u',v'\in M'.
\end{align*}
\end{prop}
\begin{proof}
By \eqref{6-6} and \eqref{morphism}, we have
\begin{align*}
\psi(u\circ_\la v)&=\psi(T(u)_\la v)=(\chi\circ T)(u)_\la\psi(v){=}(T^\prime\circ\psi)(u)_\la\psi(v)=\psi(u)\circ'_\la\psi(v),\\
\psi(u\vee_\la v)&=\psi\phi_\la(T(u),T(v)){=}\phi^\prime_\la(\chi\circ T(u),\chi\circ T(v)){=}\phi^\prime_\la(T^\prime\circ \psi(u),T^\prime\circ \psi(v))=\psi(u)\vee'_\la\psi(v),
\end{align*}
for any $u, v\in M$. Then the result follows from Theorem \ref{th7}.
\end{proof}

\begin{prop}
Let $(\mathcal{A},\circ_\la, \vee_\la)$ be an $\NS$-Lie conformal algebra. Define
\begin{align*}
\phi_\la(a,b)=a\vee_\la b, \ \forall \ a,b\in \mathcal{A}.
\end{align*}
Then $\phi$ is a $2$-cocycle of the Lie conformal algebra $\mathcal{A}_{Lie}$ with coefficients in the module $\A$ given by \eqref{eq:module}. Furthermore, the identity map ${\rm Id}: \A\rightarrow \A_{Lie}$ is a $\phi$-twisted relative Rota-Baxter operator.
\end{prop}
\begin{proof}By a direct calculation, we have
\begin{align*}
({\mathbf{d} \phi})_{\lambda,\mu}	(a,b,c)=&a\vee_\la [b_\mu c]-[a_\la b]\vee_{\la+\mu} c-b\vee_\mu [a_\la c]+a\circ _\la(b\vee_\mu c)\\
&-b\circ_\mu (a\vee_\la c)+c\circ_{-\la-\mu-\p}(a\vee_\la b)\stackrel{\eqref{NS2}}{=}0,
\end{align*}
which proves that $\phi$ is a $2$-cocycle.  The rest follows from \eqref{NS5-5}.
\end{proof}


\begin{prop}\label{p6}  Let $(\A,[\cdot_\la\cdot])$ be a Lie conformal algebra with a Nijenhuis operator $N$.
Define 
\begin{align*}
a\circ_\lambda b=[N(a)_\lambda b],~~ a\vee_\lambda b=-N[a_\lambda b], \ \forall \ a,b\in\A.
\end{align*}
Then $(\A,\circ_\lambda, \vee_\lambda)$ is an $\NS$-Lie conformal algebra.
\end{prop}
\begin{proof}
By Example \ref{Nijenhuis-Rota}, the identity map ${\rm Id}: \A\rightarrow \A^N$ is a $\phi$-twisted relative Rota-Baxter operator, where the $\lambda$-action of $\A^N$ on $\A$ and $\phi$ are given by
$$\rho(a)_\la x=[N(a)_\la x],\quad \phi_\la (a,b)=-N[a_\la b],\quad\forall~a,b\in\A^N,x\in\A.$$
By Theorem \ref{th7}, the $\la$-operations $\circ_\lambda$ and $\vee_\lambda$ defined by
\begin{eqnarray*}
a\circ_\lambda b={\rm Id}(a)_\la b=[N(a)_\lambda b],\quad a\vee_\lambda b=-N[{\rm Id}(a)_\la {\rm Id}(b)]=-N[a_\lambda b]
\end{eqnarray*}
make $(\A,\circ_\lambda, \vee_\lambda)$ into an NS-Lie conformal algebra.
\end{proof}

Similar to the properties of Nijenhuis operators on Lie algebras given in \cite{KM}, we also have
\begin{prop}\label{lem:Niejproperty}
Let $(\A,[\cdot_\lambda\cdot])$ be a Lie conformal algebra with a Nijenhuis operator $N$. For all $k,l\in\Nat$, we have
\begin{itemize}
\item[$\rm(i)$]$(\A,[\cdot_\lambda\cdot]_{N^k})$ is a Lie conformal algebra;
\item[$\rm(ii)$]$N^l$ is also a Nijenhuis operator on the Lie conformal algebra $(\A,[\cdot_\lambda\cdot]_{N^k})$;
\item[$\rm(iii)$]The Lie conformal algebras $(\A,([\cdot_\lambda\cdot]_{N^k})_{N^l})$ and $(\A,[\cdot_\lambda\cdot]_{N^{k+l}})$ coincide;
\item[$\rm(iv)$]The Lie conformal algebras $(\A,[\cdot_\lambda\cdot]_{N^k})$ and $(\A,[\cdot_\lambda\cdot]_{N^l})$ are
compatible, that is,
any linear combination of $[\cdot_\lambda\cdot]_{N^k}$ and $[\cdot_\lambda\cdot]_{N^l}$ still
makes $\A$ into a Lie conformal algebra;
\item[$\rm(v)$]$N^l$ is a Lie conformal algebra homomorphism from $(\A,[\cdot_\lambda\cdot]_{N^{k+l}})$ to $(\A,[\cdot_\lambda\cdot]_{N^k})$.
\end{itemize}
\end{prop}
\begin{proof}
It follows by a straightforward calculation. We omit the details.
\end{proof}

By Propositions \ref{p6} and \ref{lem:Niejproperty}, we have
\begin{coro}
Let $(\A,[\cdot_\lambda\cdot])$ be a Lie conformal algebra and $N$ a Nijenhuis operator on $\A$. For any $k,l\in\Nat$, $(\A,\circ^{k,l}_\lambda, \vee^{k,l}_\lambda)$ is an $\NS$-Lie conformal algebra, where $\circ^{k,l}_\lambda$ and $\vee^{k,l}_\lambda$ are given by
\begin{align*}
a\circ^{k,l}_\lambda b=[N^k(a)_\lambda b]_{N^l},~~ a\vee^{k,l}_\lambda b=-N^k[a_\lambda b]_{N^l}, ~~\mbox{for}~ a,b\in\A.
\end{align*}
\end{coro}

\begin{ex}
Let $\A=\A_1\bowtie\A_2$ be a twilled Lie conformal algebra. If ${\rm p}_1$ and ${\rm p}_2$ are the corresponding projections of $\A$ onto $\A_1$ and $\A_2$, respectively, then any linear combination of ${\rm p}_1$ and ${\rm p}_2$ is a Nijenhuis operator on $\A$. Furthermore, for any $k,l\in\Nat$, $(\A,\circ^{k,l}_\lambda, \vee^{k,l}_\lambda)$ is an $\NS$-Lie conformal algebra, where $\circ^{k,l}_\lambda$ and $\vee^{k,l}_\lambda$ are given by
\begin{align*}
a\circ^{k,l}_\lambda b=k[{\rm p}_1(a)_\lambda b]+l[{\rm p}_2(a)_\lambda b],~~ a\vee^{k,l}_\lambda b=-k{\rm p}_1[a_\lambda b]-l{\rm p}_2[a_\lambda b], ~~\mbox{for}~ a,b\in\A.
\end{align*}
\end{ex}

\begin{ex}
Let $\A=\A_1\bowtie\A_2$ be a twilled Lie conformal algebra.  For any $k_1,k_2\in\C$, define $N:\A\rightarrow \A$ by
$N\mid_{\A_i}=k_i~{\rm id}_{\A_i},~i=1,2.$
Then $N$ is a Nijenhuis operator of $\A$, and $(\A,\circ_\lambda, \vee_\lambda)$ is an $\NS$-Lie conformal algebra, where $\circ_\lambda$ and $\vee_\lambda$ are given by
\begin{align*}
a\circ_\lambda b=k_1[{\rm p}_1(a)_{\lambda} b]+k_2[{\rm p}_2(a)_{\lambda} b],~~ a\vee_\lambda b=-k_1{\rm p}_1[a_\lambda b]-k_2{\rm p}_2[a_\lambda b],~~\mbox{for}~ a,b\in\A.
\end{align*}
\end{ex}

\section{Cohomology and deformations of twisted relative Rota-Baxter operators}
In this section, we give the cohomology of  twisted relative Rota-Baxter operators and use this cohomology to study infinitesimal deformations of twisted relative Rota-Baxter operators.
\subsection{Cohomology of twisted relative Rota-Baxter operators}
Let $T$ be a $\phi$-twisted relative Rota-Baxter operator on a module $(M;\rho)$ over a Lie conformal algebra $\A$. We have shown that $M^{T,\phi}=(M,[\cdot_\lambda\cdot]^{T,\phi})$ is a Lie  conformal algebra, where the $\lambda$-bracket $[\cdot_\lambda\cdot]^{T,\phi}$ is given by \eqref{eq:LCA2}. Furthermore, $\rho^T:M\rightarrow {\rm Cend}(\A)$ defined by \eqref{eq:T-Rep2} gives a module of $M^{T,\phi}$. Therefore, we obtain a cohomology  for the Lie
conformal algebra $M^{T,{\phi}}$ with coefficients in  the module $(\A;\rho^T)$.

More precisely, let $C^{0}(M,\A)=\A/\partial \A$ and for $k\geq1$, let $C^{k}(M,\A)$ be the space of $\C$-linear maps from   $M^{\otimes k}$ to $ \C[\lambda_{1},\cdots,\lambda_{k-1}]\otimes \A$ satisfying \eqref{eq:coboundary1}- \eqref{eq:coboundary3}. Denote by $C^*(M,\A)=\oplus _{k\in\Z_+}C^{k}(M,\A)$. The coboundary operator $\dM_{T}:C^0(M,\A)\rightarrow C^{1}(M,\A)$ is given by
\begin{align}\label{*}
({\mathbf{d}_T \bar{a}})(m)=\rho^T(m)_{-\partial}a=[T(m)_{-\partial} a]+T\big(l_{-\partial}(a, m)-\phi_{-\partial}(T(m),a)\big),
\end{align}
where $\bar{a}\in\A/\partial \A$, $m\in M$, and $l_{\la}(a, m)=\rho(a)_{-\lambda-\partial} m$. For $f\in C^k(M,\A)$ with $k\geq 1$,  $\mathbf{d}_Tf\in C^{k+1}(M,\A)$ is given  by
\begin{eqnarray*}
&&(\dM_T f)_{\la_1,\cdots,\la_k}(m_1,\cdots,m_{k+1})\\
&=&\sum_{i=1}^{k}(-1)^{i+1}[T(m_i)_{\lambda_{i}}f_{\la_1,\cdots,\hat{\la_i},\cdots,\la_k}(m_1,\cdots,\hat{m_i},\cdots,m_{k+1})]\\ &&+\sum_{i=1}^{k}(-1)^{i+1}T\big(\rho(f_{\la_1,\cdots,\hat{\la_i},\cdots,\la_k}(m_1,\cdots,\hat{m_i},\cdots,m_{k+1})_{-\lambda_i-\partial}m_i\big)\\
&&+\sum_{i=1}^{k}(-1)^{i+1}T\phi_{\lambda_i}\big(T(m_i),f_{\la_1,\cdots,\hat{\la_i},\cdots,\la_k}(m_1,\cdots,\hat{m_i},\cdots,m_{k+1})\big)
+\sum_{i,j=1,i<j}^{k}(-1)^{k+i+j+1} \\&& \times f_{{\lambda_1,\cdots,\hat{\la}_i,\cdots,\hat{\la}_j,\cdots,\la_k,\la^\dag_{k+1}}}\big(m_1,\cdots,m_k,(T(m_i)_{\lambda_i}m_j-T(m_j)_{-\lambda_i-\partial}m_i+\phi_{\lambda_i}(T(m_i),T(m_j)))\big)\\
&&+(-1)^{k}[T(m_{k+1})_{\lambda^\dag_{k+1}}f_{\lambda_1,\cdots,\lambda_{k-1}}(m_1,\cdots,m_k)]+(-1)^{k}T\big(\rho(f_{\lambda_1,\cdots,\lambda_{k-1}}(m_1,\cdots,m_k))_{-\lambda^\dag_{k+1}-\partial}m_{k+1}\big)\\
&&+(-1)^{k}T\phi_{\lambda^\dag_{k+1}}\big(T(m_{k+1}),f_{\lambda_1,\cdots,\lambda_{k-1}}(m_1,\cdots,m_k)\big) +\sum_{i=1}^{k}(-1)^{i}\\&&\times f_{\lambda_1,\cdots,\hat{\la_i},\cdots,\la_k}\big(m_1,\cdots,\hat{m}_i,\cdots,m_k,T(m_{i})_{\lambda_{i}}m_{k+1}+ T(m_{k+1})_{-\lambda_{i}-\partial}m_i+\phi_{\lambda_i}(T(m_i),T(m_{k+1})) \big),
\end{eqnarray*}
where $m_1,\cdots,m_{k+1}\in M$ and $\lambda_{k+1}^{\dag}=-\sum_{j=1}^{k}\lambda_{j}-\partial$.

\begin{defi}{
Let $T:M\rightarrow \A$ be a $\phi$-twisted relative Rota-Baxter operator. The cohomology of the cochain complex $(C^*(M,\A),\mathbf{d}_T)$ is called the {\bf cohomology of the $\phi$-twisted relative Rota-Baxter operator}. The corresponding $k$-th cohomology group, which we denote by $H_T(M,\A)$, is called the {\bf $k$-th cohomology group} for the $\phi$-twisted relative Rota-Baxter operator $T$.}
\end{defi}
\begin{remark}
Since a relative Rota-Baxter operator is a $\phi$-twisted relative Rota-Baxter operator with $\phi=0$, the above cohomology of $\phi$-twisted relative Rota-Baxter operators recovers the cohomology of relative Rota-Baxter operators in the case of $\phi=0$.
\end{remark}

Let $T:M\rightarrow \A$ be a $\phi$-twisted relative Rota-Baxter operator. We have shown in Proposition \ref{pro:twisted L_infty} that  $(C^*(M,\A),l_1^T,l_2^T,l_3^T)$ is an $L_{\infty}$-algebra with trivial higher brackets. Therefore, $l_1^T\circ l_1^T=0$. Furthermore, the following result holds.
\begin{prop}
For $f\in C^k(M,\A)$ with $k\geq 1$, we have $\dM_{T} f=(-1)^{k}l_1^T(f).$
\end{prop}
\begin{proof}
It follows by a direct calculation. We omit the details.
\end{proof}

\subsection{Infinitesimal deformations of twisted relative Rota-Baxter operators}
\begin{defi} Let $(M;\rho)$ be a module over a Lie conformal algebra $\A$ and
$T:M\rightarrow\A$ a $\phi$-twisted relative Rota-Baxter operator.  An {\bf infinitesimal deformation} of $T$ is a $t$-parameterized sum $T_t=T+t\T$ for some $\T\in C^{1}(M,\A)$ such that, for $m,n\in \A$, the following equation holds:
\begin{eqnarray*}\label{eq:def RB}
[T_t(m)_\lambda T_t(n)]\equiv T_t\big(\rho(T_t(m))_\lambda n-\rho(T_t(n))_{-\la-\partial}m+\phi_\lambda (T_t(m),T_t(n))\big) ~~(\mbox{{\rm mod}}~t^2).
\end{eqnarray*}
In this case, we also say that {\bf $\T$ generates an infinitesimal deformation} of $T$.
\end{defi}

It is straightforward to check that $\T$ generates an infinitesimal deformation of $T$ if and only if for $m,n\in M$,  the following condition holds:
\begin{equation}\label{4-5}
\begin{split}
[T(m)_\la \T(n)]+&[\T(m)_\la T(n)]=\T\big(\rho(T(m))_\la n-\rho(T(n))_{-\la-\partial}m+\phi_\la(T(m),T(n))\big)\\
+&T\big(\rho(\T(m))_\la n-\rho(\T(n))_{-\la-\partial} m+\phi_\la(\T(m),T(n))+\phi_\la(T(m),\T(n))\big).
\end{split}
\end{equation}
This is equivalent to say that $\T$  is a $1$-cocycle in the cohomology of $T$, namely, $\mathbf{d}_T(\T)=0.$

\begin{defi}Two infinitesimal deformations $T_t=T+t\T$ and $T^\prime_t=T+t\T^\prime$ of a $\phi$-twisted relative Rota-Baxter operator $T$ are said to be {\bf equivalent} if there exists an element $\bar{a}\in\A/\partial\A$ such that the pair maps $(\chi_t,\psi_t)$ defined by
\begin{align*}
\chi_t(x)=x-t[x_{-\partial}a], \ \ \psi_t(m)=m+tl_{-\partial}(a, m)-t\phi_{-\partial}(T(m),a),
\end{align*}
satisfy
\begin{equation*}\label{morphism2}
\chi_t\circ T_t\equiv T_t^\prime\circ \psi_t,\ \
\rho(\chi_t(a))_\la \psi_t(m)\equiv \psi_t(\rho(a)_\la m),\ \
\psi_t\circ\phi_\la(a,b)\equiv \phi_\la(\chi_t(a),\chi_t(b)),
\end{equation*}
for $a,b\in\A$ and $m\in M$, where $``\equiv"$ means ``equality under modulo $ t^2$".
\end{defi}

Notice that  $(\chi_t,\psi_t)$  is well defined since, if $\partial a\in \partial\A$,  $[x_{-\partial}\partial a]$, $l_{-\partial}(\partial a, m)$ and $\phi_{-\partial}(T(m),\partial a)$  are zero due to conformal sesquilinearity. An infinitesimal deformation $T_t=T+t\T$ of  $T$ is said to be {\bf trivial} if $T_t$ is equivalent to $T$.

Now suppose that $T_t=T+t\T$ and $T^\prime_t=T+t\T^\prime$  are equivalent infinitesimal deformations of $T$. By the condition $\chi_t\circ T_t\equiv T_t^\prime\circ \psi_t~(\mbox{mod}~ t^2)$, we have
\begin{align}
\T(m)-\T^\prime(m)&=[T(m)_{-\partial}a]+T\big(l_{-\partial}(a, m)-\phi_{-\partial}(T(m),a)\big).\label{4-3-6}
\end{align}
By the condition $\rho(\chi_t(x))_\la \psi_t(m)\equiv\psi_t(\rho(x)_\la m)~(\mbox{mod}~ t^2)$, we have
\begin{align}
&\rho(x)_\la \phi_{-\partial}(T(m),a)=\phi_{-\partial}(T(\rho(x)_\la m),a).\label{4-3-3}
\end{align}
By the condition $\psi_t\circ\phi_\la(x,y)\equiv\phi_\la(\chi_t(x),\chi_t(y))~(\mbox{mod}~ t^2)$, we have
\begin{align}
l_{-\partial}(a,\phi_\la(x,y))&=\phi_{-\partial}(T\phi_\la(x,y),a)-\phi_\la(x,[y_{-\partial}a])-\phi_\la([x_{-\partial}a],y).\label{4-3-4}
\end{align}

It follows from \eqref{4-3-6} that
\begin{align}
\T(m)-\T^\prime(m)=(\mathbf{d}_T~ \bar{a})(m), \quad\mbox{for} \ m\in M.
\end{align}
Hence we have the following result:
\begin{theo}
Let $T_t=T+t\T$ and $T^\prime_t=T+t\T^\prime$ be two equivalent infinitesimal deformations of the   $\phi$-twisted relative Rota-Baxter operator $T$. Then $\T$ and $\T'$  are in the same cohomology class  $ H_{T}^1(M,\A) $.
\end{theo}

\begin{defi}{ Let $T$ be a $\phi$-twisted relative Rota-Baxter operator on a module  $(M;\rho)$ over a Lie conformal algebra $\A$. An element $\bar{a}\in \A/\partial \A$ is called a {\bf Nijenhuis element} associated to $T$ if the representative element $a$ satisfies \eqref{4-3-3} and \eqref{4-3-4}.
}\end{defi}

Denote by ${\rm Nij}(T)$ the set of Nijenhuis elements associated to the $\phi$-twisted relative Rota-Baxter operator $T$. It is easy to see from \eqref{4-3-6}--\eqref{4-3-4} that a trivial infinitesimal deformation  of $T$ produces a Nijenhuis element. Notably, the converse is also valid, as the following theorem shows.

\begin{theo}
Let $T$ be a $\phi$-twisted relative Rota-Baxter operator on a module  $(M;\rho)$ over a Lie conformal algebra $\A$. For any $\bar{a}\in {\rm Nij}(T)$, $T_t=T+t\mathfrak{T}$ with $\mathfrak{T}=\mathbf{d}_T(\bar{a})$ is a trivial infinitesimal deformation of $T$.
\end{theo}
\begin{proof} As $\mathfrak{T}=\mathbf{d}_T (\bar{a})$, $\mathbf{d}_T(\T)=0$ and thus condition \eqref{4-5} is valid. Evidently,  \eqref{4-3-6}--\eqref{4-3-4} are satisfied and therefore $\T $ generates a trivial infinitesimal deformation of $T$.
\end{proof}

\vs{10pt}
\noindent{\bf{ Acknowledgements.}}\ {This research was supported  by the National Key Research and Development Program of China (2022T150109) and the Fundamental Research Funds for the Central Universities (2022FRFK060025, 2412022QD033).}

												\end{document}